\documentclass[12pt,reqno]{amsart}

\usepackage{amsmath,amssymb,ifthen}
\usepackage{fullpage}
\usepackage{graphicx,psfrag,subfigure}
\usepackage{color}
\usepackage{moreverb}
\usepackage{amsthm}
\usepackage{hhline}
\usepackage{bbm}
\usepackage[english]{babel}
\usepackage{tikz}
\usetikzlibrary{calc,math}
\usepackage{pgfplots}
\usepackage[utf8]{inputenc}
\usepackage[T1]{fontenc}
\usepackage{amsfonts}
\usepackage{enumerate}
\usepackage{hyperref}

\usepackage{pgfplots}
\usepgfplotslibrary{external}
\usepgfplotslibrary{colorbrewer}

\newcommand{\logLogSlopeTriangle}[6]
{

    \pgfplotsextra
    {
        \pgfkeysgetvalue{/pgfplots/xmin}{\xmin}
        \pgfkeysgetvalue{/pgfplots/xmax}{\xmax}
        \pgfkeysgetvalue{/pgfplots/ymin}{\ymin}
        \pgfkeysgetvalue{/pgfplots/ymax}{\ymax}

        \pgfmathsetmacro{\xArel}{#1}
        \pgfmathsetmacro{\yArel}{#3}
        \pgfmathsetmacro{\xBrel}{#1-#2}
        \pgfmathsetmacro{\yBrel}{\yArel}
        \pgfmathsetmacro{\xCrel}{\xArel}

        \pgfmathsetmacro{\lnxB}{\xmin*(1-(#1-#2))+\xmax*(#1-#2)} 
        \pgfmathsetmacro{\lnxA}{\xmin*(1-#1)+\xmax*#1} 
        \pgfmathsetmacro{\lnyA}{\ymin*(1-#3)+\ymax*#3} 
        \pgfmathsetmacro{\lnyC}{\lnyA+#4*(\lnxA-\lnxB)}
        \pgfmathsetmacro{\yCrel}{\lnyC-\ymin)/(\ymax-\ymin)} 

        \coordinate (A) at (rel axis cs:\xArel,\yArel);
        \coordinate (B) at (rel axis cs:\xBrel,\yCrel);
        \coordinate (C) at (rel axis cs:\xCrel,\yCrel);

        \draw[#5]   (A)--
                    (B)-- 
                    (C)-- node[pos=0.5,anchor=west] {#6}
                    cycle;
    }
}

\newcommand{\logLogSlopeTriangleBelow}[6]
{

    \pgfplotsextra
    {
        \pgfkeysgetvalue{/pgfplots/xmin}{\xmin}
        \pgfkeysgetvalue{/pgfplots/xmax}{\xmax}
        \pgfkeysgetvalue{/pgfplots/ymin}{\ymin}
        \pgfkeysgetvalue{/pgfplots/ymax}{\ymax}

        \pgfmathsetmacro{\xArel}{#1}
        \pgfmathsetmacro{\yArel}{#3}
        \pgfmathsetmacro{\xBrel}{#1-#2}
        \pgfmathsetmacro{\yBrel}{\yArel}
        \pgfmathsetmacro{\xCrel}{\xArel}

        \pgfmathsetmacro{\lnxB}{\xmin*(1-(#1-#2))+\xmax*(#1-#2)} 
        \pgfmathsetmacro{\lnxA}{\xmin*(1-#1)+\xmax*#1} 
        \pgfmathsetmacro{\lnyA}{\ymin*(1-#3)+\ymax*#3} 
        \pgfmathsetmacro{\lnyC}{\lnyA+#4*(\lnxA-\lnxB)}
        \pgfmathsetmacro{\yCrel}{\lnyC-\ymin)/(\ymax-\ymin)} 

        \coordinate (A) at (rel axis cs:\xArel,\yArel);
        \coordinate (B) at (rel axis cs:\xBrel,\yCrel);
        \coordinate (C) at (rel axis cs:\xBrel,\yArel);

        \draw[#5]   (A)--
                    (B)-- node[pos=0.5,anchor=east] {#6}
                    (C)-- 
                    cycle;
    }
}

\let\div\relax
\DeclareMathOperator{\div}{\mathrm{div}}

\renewcommand{\d}{\,{\rm d}}

\newcommand{\const}[1]{C_{\text{\rm#1}}}

\newcommand{\set}[2]{\big\{#1\,:\,#2\big\}}

\newcommand{\dual}[3][]{#1\langle#2\,,\,#3#1\rangle}
\newcommand{\norm}[3][]{#1\|#2#1\|_{#3}}

\newcommand\A{\mathbb{A}}

\newcommand\R{\mathbb{R}}

\newcommand\OO{{\mathcal O}}

\newcommand\out[1]{}

\numberwithin{equation}{section}
\numberwithin{figure}{section}
\newtheorem{theorem}{Theorem}[section]
\newtheorem{proposition}[theorem]{Proposition}
\newtheorem{lemma}[theorem]{Lemma}

\newtheorem{remark}[theorem]{Remark}

\newcommand*\patchAmsMathEnvironmentForLineno[1]{%
  \expandafter\let\csname old#1\expandafter\endcsname\csname #1\endcsname
  \expandafter\let\csname oldend#1\expandafter\endcsname\csname end#1\endcsname
  \renewenvironment{#1}%
     {\linenomath\csname old#1\endcsname}%
     {\csname oldend#1\endcsname\endlinenomath}}%
\newcommand*\patchBothAmsMathEnvironmentsForLineno[1]{%
  \patchAmsMathEnvironmentForLineno{#1}%
  \patchAmsMathEnvironmentForLineno{#1*}}%
\AtBeginDocument{%
\patchBothAmsMathEnvironmentsForLineno{equation}%
\patchBothAmsMathEnvironmentsForLineno{align}%
\patchBothAmsMathEnvironmentsForLineno{flalign}%
\patchBothAmsMathEnvironmentsForLineno{alignat}%
\patchBothAmsMathEnvironmentsForLineno{gather}%
\patchBothAmsMathEnvironmentsForLineno{multline}%
}
\usepackage[mathlines]{lineno}

\renewcommand{\A}{{\boldsymbol{A}}}
\newcommand{\n}{{\boldsymbol{n}}}
\renewcommand{\u}{{\boldsymbol{u}}}
\renewcommand{\v}{{\boldsymbol{v}}}
\newcommand{\x}{{\boldsymbol{x}}}
\newcommand{\y}{{\boldsymbol{y}}}

\newcommand{\ssigma}{{\boldsymbol{\sigma}}}
\newcommand{\ttau}{{\boldsymbol{\tau}}}

\title{Space-time FEM-BEM couplings\\ for parabolic transmission problems}

\author{Thomas~F\"uhrer}
\address{Facultad de Matemáticas, Pontificia Universidad Católica de Chile, Santiago, Chile}
\email{thfuhrer@uc.cl}

\author{Gregor~Gantner}
\address{Institute for Numerical Simulation, University of Bonn, Bonn, Germany}
\email{gantner@ins-uni.bonn.de}

\author{Michael~Karkulik}
\address{Departamento de Matemática, Universidad Técnica Federico Santa María, Valparaíso, Chile}
\email{michael.karkulik@usm.cl}


\begin{document}

\begin{abstract}
We develop couplings of a recent space-time first-order system least-squares (FOSLS) method for parabolic problems and space-time boundary element methods (BEM) for the heat equation 
to numerically solve a parabolic transmission problem on the full space and a finite time interval. 
In particular, we demonstrate coercivity of the couplings under certain restrictions and validate our theoretical findings by numerical experiments.
\end{abstract}

\date{\today}
\maketitle

\section{Introduction}

\subsection{Model problem}
For $d\in\{2,3\}$, let $\Omega\subset\R^d$ be a bounded Lipschitz domain with connected boundary $\Gamma$ and $T>0$ a given
final time with corresponding time interval $I:=(0,T)$.
We abbreviate the space-time cylinder by $Q:=I\times\Omega$, and denote its lateral boundary
by $\Sigma:=I\times\Gamma$. The corresponding outward normal vector on $\Sigma$ is denoted by
$\n^\top=(0,\n_\x^\top)\in\R^{d+1}$. 
Moreover, let $\Omega^{\rm ext}:=\R^d\setminus\overline\Omega$ and $Q^{\rm ext}:=I\times\Omega^{\rm ext}$ be the corresponding exterior domains. 
For $\A\in L^\infty(Q)^{d\times d}$ symmetric uniformly positive definite and given data $f: Q \to \R$, $u_0: \Omega \to \R$, $u_D: \Sigma \to \R$, and $\phi_N: \Sigma \to \R$,
whose properties will be specified below, we consider the parabolic transmission problem to find $u:Q\rightarrow\R$ and
$u^{\rm ext}:Q^{\rm ext}\rightarrow\R$ such that
\begin{subequations} \label{eq:transmission}
\begin{alignat}{2} 
	\partial_t u - \div_\x (\A\nabla_\x u) & =  f &\quad& \text{ in } Q, \label{eq:heat_interior}\\
	\partial_t u^{\rm ext} - \Delta_\x u^{\rm ext} & = 0 && \text{ in } Q^{\rm ext}, \label{eq:heat_exterior}\\
	u(0,\cdot) & =  u_0 && \text{ in }\Omega, \label{eq:initial_interior}\\
	u^{\rm ext}(0,\cdot) & =  0 && \text{ in }\Omega^{\rm ext}, \label{eq:initial_exterior}\\
	u - u^{\rm ext} & =  u_D && \text{ on }\Sigma, \label{eq:jump_dirichlet}\\
	(\A\nabla_\x u - \nabla_\x u^{\rm ext}) \cdot \n_\x & =  \phi_N && \text{ on }\Sigma \label{eq:jump_neumann},\\
	u^{\rm ext}(\cdot,\x) &\to 0  &&\text{ on } I \text{ as }|\x|\to\infty \label{eq:radiation}.
\end{alignat}
\end{subequations}
Here, the convergence in~\eqref{eq:radiation} is to be understood uniformly for $t\in I$ and $\x$ with $|\x|\to\infty$.

\subsection{Motivation and existing works}
Problem~\eqref{eq:transmission} models, for instance, heat conduction in a heterogeneous body
$\Omega$ that radiates heat into the surrounding homogeneous medium~$\Omega^{\rm ext}$. 
In this context, the unbounded domain $\Omega^{\rm ext}$ can of course be viewed as a simplification of a surrounding \emph{bounded} $\widetilde\Omega$ that is substantially larger than $\Omega$, 
and we refer for instance to \cite{cz98,hllz13,fr17,guo21,bdfgl24} for parabolic transmission problems on bounded domains.
The work~\cite{rs07} provides a concrete interpretation of~\eqref{eq:transmission} in terms of photothermal
spectroscopy. 
Other works that investigate the heat equation on unbounded domains include for instance~\cite{dubach96,ck97,hh02,lg07,sd11,rsm20,zx23}. 
Assuming that $\A$ is scalar and constant on $Q$, $f=0$ on $Q$, and $u_0=0$ on $\Omega$,
\cite{qrsz19} discretize \eqref{eq:transmission} via a boundary element method (BEM) in the interior
and the exterior, and the resulting semidiscrete system is solved by
convolution quadrature \cite{lubich88,ls92,lo93} in time.
In the PhD thesis~\cite{dohr19}, a coupling of the space-time finite element method (FEM)~\cite{sz20} based on some Hilbert-type transformation and space-time BEM~\cite{an87,noon88,costabel90} is proposed, and a numerical example for $d=1$ is presented. 
However, no numerical analysis of the scheme is given, and it is unclear whether the resulting discrete systems are inf-sup stable or even uniquely solvable. 
A provably stable numerical scheme that approximately solves~\eqref{eq:transmission} in its general form has been
missing in the literature up to now.
It is the objective of this manuscript to introduce and analyze such a method.
\medskip

A class of problems related to~\eqref{eq:transmission}, so-called parabolic-elliptic
transmission problems in which \eqref{eq:heat_exterior} and \eqref{eq:radiation} for $u^{\rm ext}$ are
replaced by the Laplace equation with appropriate radiation condition,
have been intensely studied, e.g., in \cite{ms87,ces90,costabel90,ees18,es20}. 
As a matter of fact, \cite{ms87} models electromagnetic eddy-current problems as
parabolic-hyperbolic transmission problems, where \eqref{eq:heat_exterior} is actually replaced
by a wave equation with small, potentially negligible, coefficient in front of the
involved second time derivative of $u^{\rm ext}$.

For transmission problems with stationary elliptic PDEs in the interior and the exterior, FEM-BEM coupling is
a well-understood numerical method,
see, e.g., the overview article~\cite{affkmp13} and the references therein. 
We only mention the recent works~\cite{fhk17,fh17}, which couple a discontinuous Petrov--Galerkin (DPG) method with standard as well as least-squares boundary element methods. 
Indeed, the couplings for the parabolic transmission problem~\eqref{eq:transmission} which we will introduce in the present manuscript were inspired by the couplings of those works. 

\subsection{Proposed couplings}
We first reformulate the interior problem \eqref{eq:heat_interior}\& \eqref{eq:initial_interior} as first-order system by substituting $\ssigma:=-\A\nabla_\x u$. 
For the heat equation, i.e., $\A={\bf Id}$, it has only recently been shown in~\cite{fk21}
that the resulting first-order system least-squares (FOSLS) formulation,
equipped with homogeneous Dirichlet boundary conditions for $u$, is well posed. 
This result was extended to general diffusion matrices $\A$
and inhomogeneous Dirichlet and/or Neumann boundary conditions in~\cite{gs21}.
The results of~\cite{gs21} will be crucially exploited in the present work, but
we also mention the recent generalizations \cite{gs22,fgk23} to the
instationary Stokes problem and the wave equation, respectively.

In a second step, we couple the FOSLS formulations \cite{fk21,gs21} of the interior problem
\eqref{eq:heat_interior}\&\eqref{eq:initial_interior} via the jump
conditions~\eqref{eq:jump_dirichlet}\&\eqref{eq:jump_neumann} with boundary integral
formulations \cite{an87,noon88,costabel90} of the exterior problem
\eqref{eq:heat_exterior}\&\eqref{eq:initial_exterior}\&\eqref{eq:radiation}.
While the exterior problem is posed on the \emph{unbounded} $(d+1)$-dimensional set $Q^{\rm ext}$,
its boundary integral reformulation is now posed on the \emph{bounded} $d$-dimensional set $\Sigma$.
We consider couplings of the FOSLS formulations to three different
boundary integral formulations of the exterior problem: (i) the weakly-singular integral equation
(with unknown Neumann data), (ii) the hypersingular integral equation (with unknown Dirichlet data),
and (iii) the sum of the two preceding equations.
Similar boundary integral formulations are also possible for the heat equation in the interior
in case of homogeneous source term $f=0$ and homogeneous initial data $u_0=0$,
which again yields a dimension reduction by one. 
Besides this dimension reduction, the space-time boundary integral equations
of~\cite{an87,noon88,costabel90} have the advantage that the corresponding variational
formulations are \emph{always} coercive.

We prove coercivity of the resulting variational formulations of the three coupling methods,
but we require the restriction $\A={\bf Id}$ for the first two of them.
The reason for this is that we exploit a specific relation between the Dirichlet and Neumann
data from the interior and the exterior problem in terms of boundary integral operators, which is
only valid if we consider the same PDE in the interior and the exterior. 
In the stationary case, \cite{sayas09,steinbach11} removed this restriction for the
Johnson--N\'ed\'elec coupling, allowing for spatial diffusion matrices in the 
interior with minimal eigenvalue strictly larger than $1/4$. 
The latter can even be slightly improved, cf.~\cite{os13}.
A generalization to problem~\eqref{eq:transmission} and the FEM-BEM couplings considered here
is not straight-forward. 
However, no such restriction is required for the third coupling involving all
four boundary integral operators of the weakly-singular and the hypersingular equation. 
For all three coupling methods, the proof of coercivity requires the FOSLS part
to be weighted with a sufficiently large positive number $\alpha>0$ whose size is in principle unknown. 
That being said, in our numerical experiments (for the first coupling), we always choose
$\alpha=1$ without encountering any issues. We also present numerical experiments showing
the effect of $\alpha$ on the coercivity of the resulting discrete systems.
In these experiments on the unit square, only a small weighting of the order of $0.1$ or less results in a Galerkin matrix which is not positive definite.

\subsection{Space-time discretization}
The coupling methods presented in this work are \emph{coercive} and fit thus into the setting 
of the Lax--Milgram lemma. Hence, in contrast to time-stepping methods as~\cite{qrsz19}, 
they provide quasi-optimal approximations of the solution $(u,\ssigma) = (u,-\A\nabla_\x u)$ for
\emph{arbitrary} discrete trial spaces. (The solution in the exterior $u^{\rm ext}$ can be
computed in a post-processing step via a representation formula.)
Regarding discretizations of the FOSLS methods \cite{fk21,gs21} for the interior problem with homogeneous Dirichlet boundary conditions,
  standard Lagrange elements on simplicial meshes of the full space-time cylinder $Q$ can lead
  to disappointingly low convergence rates, cf.~\cite{fk21}. However,~\cite{gs24} proposes
  discrete trial spaces on prismatic meshes with enhanced approximation properties,
requiring less regularity of the unknown exact solution $(u,\ssigma)$. This trial space
will also be used in the present work.

The fact that arbitrary discrete trial spaces are admissible for our couplings allows for a further important advantage of space-time methods,
namely the possibility of local mesh refinement simultaneously in both space and time.
We intend to investigate suitable a posteriori error estimators to adaptively
steer space-time local mesh refinement in the future.
In this regard, we remark that the least-squares functional of the space-time FOSLS
formulation~\cite{fk21,gs21} automatically provides a reliable and efficient
estimator, and that~\cite{gv22} recently introduced an estimator for the
weakly-singular boundary integral equation. 
For a posteriori error analysis of stationary FEM-BEM couplings,
we refer to~\cite{affkmp13} and the references therein. 

While a time-stepping method as~\cite{qrsz19} only needs to solve relatively
small linear systems in each time step, a potential downside of simultaneous
space-time methods~\cite{an87,noon88,costabel90,ss09,steinbach15,dns19,sz20,fk21,gs21,gs22,fgk23}
is that one has to solve one huge linear system at once. 
In addition, boundary element discretizations lead to dense matrices due to the nonlocality of the boundary integral operators.
On the other hand, simultaneous space-time methods in principle have the potential for
massive parallelization, see, e.g., \cite{gn16,ns19,vw20} for FEM
and \cite{dzomk19,zwom21} for BEM.
Preconditioners for the employed space-time FOSLS formulation~\cite{fk21,gs21} that can be applied in parallel are still missing in the literature and shall be  addressed in the future.  
Having efficient preconditioners for FEM and BEM at hand, block-diagonal preconditioners give efficient preconditioners for FEM-BEM couplings as has been analyzed in, e.g., \cite{ms98,ffps17} for elliptic transmission problems. 
Moreover, the dense BEM matrices can be compressed via the fast multipole method and $\mathcal{H}$-matrices, cf.~\cite{mst14,mst15,ht18,wmoz22}.

Finally, coercivity also allows to easily treat optimal control and/or parameter-dependent problems for~\eqref{eq:transmission} as in the works~\cite{gs23,fk23,fk24} for the space-time FOSLS formulation of parabolic PDEs~\cite{fk21,gs21}. 
Note that applications to optimal control particularly make the potential advantage of time-stepping methods requiring to store the solution only at each time step obsolete, as the solution on the full space-time cylinder is needed simultaneously for these problems anyway. 

\subsection{Outline}
In Section~\ref{sec:sobolev}, we briefly recall anisotropic Sobolev spaces on the (exterior) space-time cylinder and its lateral boundary. 
We further recall the definition and mapping properties of the corresponding boundary integral operators as well as the corresponding representation formula for solutions of the interior heat equation in Section~\ref{sec:bios}. 
Section~\ref{sec:weak_formulation} then states the precise weak formulation of the considered parabolic transmission problem~\eqref{eq:transmission}.
In particular, a representation formula for the exterior solution $u^{\rm ext}$ is derived in terms of its Cauchy data, exploiting the radiation condition~\eqref{eq:radiation}. 
Section~\ref{sec:fosls} briefly summarizes the employed results of~\cite{gs21} on space-time FOSLS formulations of second-order parabolic problems with inhomogeneous Dirichlet and Neumann boundary conditions. 
Section~\ref{sec:couplings} and~\ref{sec:coercivity}  constitute the core of the paper, introducing three different couplings of FOSLS for the interior problem with boundary integral equations for the exterior problem and proving their coercivity (under certain restrictions). 
In Section~\ref{sec:equivalence}, we verify that these couplings are indeed equivalent to the transmission problem~\eqref{eq:transmission}.
We particularly show that~\eqref{eq:transmission} actually has a unique solution.
The paper is concluded in Section~\ref{sec:numerics} with numerical experiments, demonstrating the performance of the first coupling for smooth but also the limitations for non-smooth solutions under uniform mesh refinement. 
We also briefly comment on the implementation and analyze positive definiteness
of the Galerkin matrices in dependence of the stabilization parameter $\alpha$ for the FOSLS part. 

\subsection{General notation}
Throughout and without any ambiguity, $|\cdot|$ denotes the absolute value of scalars, the Euclidean norm of vectors in $\R^n$, or the measure of a set in $\R^n$, e.g., the length of an interval or the area of a surface in $\R^3$.
We write $A\lesssim B$ to abbreviate $A\le CB$ with some generic constant $C>0$ which is clear from the context.
Moreover, $A\eqsim B$ abbreviates $A\lesssim B\lesssim A$.

\section{Preliminaries}\label{sec:preliminaries}

\subsection{Sobolev spaces} \label{sec:sobolev}
We use standard notation for Lebesgue and Sobolev spaces on bounded Lipschitz domains $\Omega$ and the corresponding boundary $\Gamma$, i.e.,
$L^2(\Omega)$, $H^1(\Omega)$, $H^{-1}(\Omega) = H^1_0(\Omega)'$ and $H^{1/2}(\Gamma)$, $H^{-1/2}(\Gamma) = H^{1/2}(\Gamma)'$.
Similarly, we write $L^2(I;X)$ for the Bochner space of all measurable and square-integrable functions mapping from the time interval $I$ to some Banach space $X$. 
Moreover, we recall the definition of Bochner--Sobolev spaces $H^\nu(I;X) := \set{v\in L^2(I;X)}{\norm{v}{H^{\nu}(I;X)} < \infty}$ for $\nu\in (0,1]$, 
associated with the norm
\begin{align*}
	\norm{v}{H^{\nu}(I;X)}^2
	:= \norm{v}{L^2(I;X)}^2 + |v|_{H^{\nu}(I;X)}^2,
	\quad  |v|_{H^{\nu}(I;X)}^2
	:=\begin{cases}
	\int_I\int_I \frac{\norm{v(t)-v(s)}{X}^2}{|t-s|^{1+2\nu}}\d s\d t &\text{ if }\nu\in(0,1),
	\\
	\norm{\partial_t v}{L^2(I;X)}^2 &\text{ if }\nu=1,
	\end{cases}
\end{align*}
where $\partial_t$ denotes the (weak) time derivative.
Finally, we define the anisotropic Sobolev spaces
\begin{align*}
	H^{1,1/2}(Q) &:= L^2(I;H^{1}(\Omega))\cap H^{1/2}(I;L^2(\Omega)),
	\\
	H^{1/2,1/4}(\Sigma)&:= L^2(I;H^{1/2}(\Gamma))\cap H^{1/4}(I;L^2(\Gamma)).
\end{align*}
We denote by $H^{-1/2,-1/4}(\Sigma)$ the dual space of $H^{1/2,1/4}(\Sigma)$ with duality pairing $\dual{\cdot}{\cdot}_{\Sigma}$ and interpret $L^2(\Sigma)$ as subspace of $H^{-1/2,-1/4}(\Sigma)$ via
\begin{align*}
	\dual{u}{\phi}_\Sigma := \dual{\phi}{u}_\Sigma := \int_\Sigma u(t,\x) \phi(t,\x) \d\x \d t
	\quad\text{for all }u\in H^{1/2,1/4}(\Sigma) \text{ and }\phi\in L^2(\Sigma).
\end{align*}

We recall the continuous embedding
\begin{align} \label{eq:V2C0}
	L^2(I;H^1(\Omega)) \cap H^1(I; H^{-1}(\Omega)) \hookrightarrow H^{1,1/2}(Q) \cap C^0(\overline I; L^2(\Omega))
\end{align}
from, e.g., \cite[pages 480 \& 494]{dl92}. 
Next, we recall that the restriction $(\cdot)|_\Sigma$ is a bounded linear mapping from $H^{1,1/2}(Q)$ to $H^{1/2,1/4}(\Sigma)$; see, e.g., \cite[Chapter~4, Theorem~2.1]{lm72b}. 
Finally, \cite[Proposition~2.18]{costabel90} states that the normal derivative $\nabla_\x(\cdot)\cdot \n_\x$ is a bounded linear mapping from the space $\set{u\in H^{1,1/2}(Q)}{\partial_t u - \Delta_\x u \in L^2(Q)}$ equipped with the canonical graph norm to $H^{-1/2,-1/4}(\Sigma)$. 

For $R>0$, let $B_R(\boldsymbol{0}) := \set{\x\in\R^d}{|\x| < R}$ with boundary $\Gamma_R := \partial B_R(\boldsymbol{0})$. 
For $R>0$ such that $\overline\Omega \subset B_R(\boldsymbol{0})$,
we abbreviate $\Omega^{\rm ext}_R:= \Omega^{\rm ext}\cap B_R(\boldsymbol{0})$, $Q_R^{\rm ext} := I \times \Omega_R^{\rm ext}$, and $\Sigma_R := I \times \Gamma_R$ with
corresponding outward normal vector $\n_{R}^\top=(0,\n_{R,\x}^\top)$.
We define the Sobolev spaces on $Q_R^{\rm ext}$ and $\Sigma_R$ as for $Q$ and $\Sigma$.

\subsection{Boundary integral operators} \label{sec:bios}
Define the heat kernel
\begin{align*}
 G(t,\x) := \begin{cases} \frac{1}{(4\pi t)^{d/2}} \, e^{-\frac{|\x|^2}{4t}} \quad &\text{for }(t,\x)\in (0,\infty)\times\R^d,
 \\
 0 \quad &\text{else}.
 \end{cases}
\end{align*}
Note that for $(t,\x)\neq (0,\boldsymbol{0})$, the function $G$ is
smooth and satisfies the homogeneous heat equation.
For sufficiently smooth functions $\phi$ and $u$ on the lateral boundary
$\Sigma$, we define the single-layer potential 
\begin{align} 
	\label{eq:single_potential}
	(\widetilde {\mathcal V} \phi)(t,\x)
	&:=\int_\Sigma G(t-s,\x-\y) \phi(s, \y) \d\y \d s
	\quad\text{for all }(t,\x)\in Q\cup Q^{\rm ext}
\intertext{and the double-layer potential}
	\label{eq:double_potential}
	 (\widetilde {\mathcal K} u)(t,\x)
	 &:=\int_\Sigma \nabla_\y G(t-s,\x-\y) \cdot \n_\y \, u(s,\y) \d\y \d s 
	 \quad\text{for all }(t,\x)\in Q\cup Q^{\rm ext}.
\end{align}
These potentials can be extended to bounded linear mappings $\widetilde{\mathcal{V}}: H^{-1/2,-1/4}(\Sigma)\to H^{1,1/2}(Q)$ and $\widetilde{\mathcal{K}}: H^{1/2,1/4}(\Sigma) \to H^{1,1/2}(Q)$. 
For any $\phi\in H^{-1/2,-1/4}(\Sigma)$ and $u\in H^{1/2,1/4}(\Sigma)$, $\widetilde {\mathcal V} \phi$ and $\widetilde {\mathcal K} u$ are indeed \emph{potentials} in the sense that they satisfy the homogeneous heat equation. 
This also readily implies that $\widetilde {\mathcal V} \phi, \widetilde {\mathcal K} u \in L^2(I;H^1(\Omega) \cap H^1(I;H^{-1}(\Omega))$. 
Finally, the restriction of $\widetilde {\mathcal V} \phi$ and $\widetilde {\mathcal K} u$ to $\{0\}\times \Omega$ is zero, i.e., $(\widetilde {\mathcal V} \phi)(0,\cdot) = (\widetilde {\mathcal K} u)(0,\cdot) = 0$.
Conversely, if $u\in L^2(I; H^1(\Omega))\cap H^1(I;H^{-1}(\Omega))$ is a weak solution of the heat equation $\partial_t u - \Delta_\x u = 0$ with $u(0,\cdot) = 0$, then there holds the  representation formula
\begin{align}\label{eq:representation}
	u = \widetilde {\mathcal V}(\nabla_\x u \cdot \n_\x)  - \widetilde {\mathcal K}(u|_\Sigma) 
	\quad \text{on }Q.
\end{align}

We further define the single-layer operator $\mathcal{V}$ and the double-layer operator $\mathcal{K}$ as in~\eqref{eq:single_potential}--\eqref{eq:double_potential} with $(t,\x)\in I\times \Sigma$ instead of $(t,\x) \in Q$. 
Note that the corresponding integral are now (weakly) singular.
These operators on the lateral boundary $\Sigma$ can be extended to bounded linear mappings $\mathcal{V}: H^{-1/2,-1/4}(\Sigma)\to H^{1/2,1/4}(\Sigma)$ and $\mathcal{K}: H^{1/2,1/4}(\Sigma)\to H^{1/2,1/4}(\Sigma)$.
In addition, $\mathcal{V}$ is coercive, i.e.,
\begin{align}\label{eq:V_coercive}
	\dual{\mathcal{V}\phi}{\phi}_\Sigma \gtrsim \norm{\phi}{H^{-1/2,-1/4}(\Sigma)}^2 
	\quad \text{for all }\phi\in H^{-1/2,-1/4}(\Sigma).
\end{align}
The restriction of the single-layer potential from the interior $Q$ and from the exterior $Q^{\rm ext}$ to $\Sigma$ coincides with the single-layer operator, i.e., 
\begin{align} \label{eq:single_operator}
	\big((\widetilde{\mathcal{V}} \phi)|_{Q}\big)|_\Sigma 
	= \big((\widetilde{\mathcal{V}} \phi)|_{Q^{\rm ext}}\big)|_\Sigma 
	= \mathcal{V} \phi.
\end{align}
The restriction of the double-layer potential from the interior $Q$ and from the exterior $Q^{\rm ext}$ to $\Sigma$ is given by
\begin{align} \label{eq:double_operator}
	\big((\widetilde{\mathcal{K}}u) |_Q\big)|_\Sigma = (\mathcal{K} - \tfrac12) u
	\quad\text{and}\quad
	\big((\widetilde{\mathcal{K}}u) |_{Q^{\rm ext}}\big)|_\Sigma = (\mathcal{K} + \tfrac12) u.
\end{align}
Finally, for sufficiently smooth functions $\phi$ and $u$ on the lateral boundary $\Sigma$, we define the boundary integral operators
\begin{align} 
	(\mathcal{N} \phi)(t,\x)
	&:=\int_\Sigma \nabla_\x G(t-s,\x-\y)\cdot\n_\x\, \phi(s, \y) \d\y \d s
	\quad\text{for all }(t,\x)\in \Sigma,
	\\ \label{eq:W}
	 ({\mathcal W} u)(t,\x)
	 &:=-\n_\x \cdot \nabla_\x \int_\Sigma \nabla_\y G(t-s,\x-\y) \cdot \n_\y)  u(s,\y) \d\y \d s
	 \quad\text{for all }(t,\x)\in \Sigma.
\end{align}
The precise meaning of~\eqref{eq:W} is given in \eqref{eq:hypsing_operator}.
These operators can be extended to bounded linear mappings $\mathcal{W}:H^{1/2,1/4}(\Sigma)\to H^{-1/2,-1/4}(\Sigma)$ and $\mathcal{N}:H^{-1/2,-1/4}(\Sigma)\to H^{1/2,1/4}(\Sigma)$. 
In addition, $\mathcal{W}$ is coercive, i.e.,
\begin{align}\label{eq:W_coercive}
	\dual{\mathcal{W}u}{u}_\Sigma \gtrsim \norm{u}{H^{1/2,1/4}(\Sigma)}^2 
	\quad \text{for all }u\in H^{1/2,1/4}(\Sigma).
\end{align}
As a matter of fact, \cite[Theorem~3.11]{costabel90} even states the following coercivity
\begin{align}\label{eq:Calderon_coercive}
\begin{split}
	(\phi,u) 
	\begin{pmatrix}\mathcal{V} & -\mathcal{K} \\ \mathcal{N} & \mathcal{W} \end{pmatrix} 
	\begin{pmatrix}\phi \\ u \end{pmatrix}
	\gtrsim \norm{\phi}{H^{-1/2,-1/4}(\Sigma)}^2 + \norm{u}{H^{1/2,1/4}(\Sigma)}^2
	\\
	\text{for all } (\phi,u)\in H^{-1/2,-1/4}(\Sigma)\times H^{1/2,1/4}(\Sigma),
\end{split}
\end{align}
which immediately yields \eqref{eq:V_coercive} and \eqref{eq:W_coercive}.
The interior and the exterior normal derivative of the single-layer potential is given by
\begin{align}\label{eq:adjdouble_operator}
	\nabla_\x\big((\widetilde{\mathcal{V}}\phi)|_Q\big) \cdot\n_x = (\mathcal{N} + \tfrac12)\phi
	\quad\text{and}\quad
	\nabla_\x\big((\widetilde{\mathcal{V}}\phi)|_{Q^{\rm ext}}\big)\cdot\n_x = (\mathcal{N}-\tfrac12)\phi.
\end{align}
The interior and exterior normal derivative of the double-layer potential coincide with $-\mathcal{W}$, i.e., 
\begin{align}\label{eq:hypsing_operator}
	\nabla_\x \big((\widetilde{\mathcal{K}} u)|_{Q}\big)\cdot \n_\x
	= \nabla_\x \big((\widetilde{\mathcal{K}} u)|_{Q^{\rm ext}}\big) \cdot \n_\x
	= -\mathcal{W} u.
\end{align}

For details on proofs and further results, we refer to the seminal works \cite{an87,noon88,costabel90}.

\subsection{Weak formulation}\label{sec:weak_formulation}
Given right-hand sides $f\in L^2(Q)$, $u_0\in L^2(\Omega)$, $u_D\in H^{1/2,1/4}(\Sigma)$, and $\phi_N \in H^{-1/2,-1/4}(\Sigma)$,
we look for weak solutions of \eqref{eq:transmission}, i.e., $u \in L^2(I;H^1(\Omega)) \cap H^1(I;H^{-1}(\Omega))$ and $u^{\rm ext}|_{Q^{\rm ext}_R} \in L^2(I;H^1(\Omega^{\rm ext}_R))$ for all $R>0$ such that $\overline\Omega \subset B_R(\boldsymbol{0})$.
In particular, the PDEs \eqref{eq:heat_interior}--\eqref{eq:heat_exterior} are to be understood in the distributional sense.
Note that $\partial_t u^{\rm ext} - \Delta_\x u^{\rm ext} = 0$ readily yields the additional regularity $u^{\rm ext}|_{Q^{\rm ext}_R} \in L^2(I;H^1(\Omega^{\rm ext}_R)) \cap H^1(I;H^{-1}(\Omega^{\rm ext}_R))$; see~\cite[Lemma 2.15]{costabel90}.
By the properties of the involved Sobolev spaces mentioned in Section~\ref{sec:sobolev}, the left-hand sides in \eqref{eq:transmission} are indeed all well defined.

Let $R>0$ be sufficiently large such that $\overline\Omega \subset B_R(\boldsymbol{0})$. 
With the associated single- and double-layer potential $\widetilde{\mathcal{V}}_R$ and $\widetilde{\mathcal{K}}_R$ on $\Sigma_R$, the representation formula~\eqref{eq:representation} applied to $Q^{\rm ext}_R$ instead of $Q$ yields that any (non-unique) weak solution of the exterior problem~\eqref{eq:heat_exterior}\&\eqref{eq:initial_exterior} is given by 
\begin{align}\label{eq:exterior_solution_R}
	u^{\rm ext} 
	= \widetilde{\mathcal{K}} (u^{\rm ext}|_\Sigma) - \widetilde{\mathcal{V}} (\nabla_\x u^{\rm ext}\cdot\n_\x) 
	+ \widetilde{\mathcal{V}}_R (\nabla_\x u^{\rm ext}\cdot\n_{R,\x}) - \widetilde{\mathcal{K}}_R (u^{\rm ext}|_{\Sigma_R}) \quad \text{on } Q_R^{\rm ext}. 
\end{align}
In what follows, we show that the radiation condition~\eqref{eq:radiation} implies that the last two
terms in~\eqref{eq:exterior_solution_R} vanish as $R\to\infty$.
While this seems to be a fundamental result, we were unable to find the precise statement in the literature and thus give a detailed proof for the sake of completeness. 
We start with the following auxiliary result, which provides a crude estimate 
for the gradient $\nabla_\x u^{\rm ext}$.

\begin{lemma}\label{lem:radiation}
Let $u^{\rm ext}$ be a weak solution of the exterior problem~\eqref{eq:heat_exterior}\&\eqref{eq:initial_exterior} with radiation condition~\eqref{eq:radiation}.
Then, for all $\epsilon>0$ there exists a constant $C(d,\epsilon,T)>0$ such that
\begin{align}
	|\nabla_\x u^{\rm ext}(t,\x)| \le C(d,\epsilon,T) \, t^{-1/2-\epsilon} \quad\text{for }t\in I \text{ as }|\x|\to\infty.
\end{align}
\end{lemma}

\begin{proof}
By parabolic regularity, e.g., \cite[Theorem~1.58]{folland95}, the homogeneous heat equation~\eqref{eq:heat_exterior} implies that $u^{\rm ext}$ is even a smooth strong solution of $\partial_t u^{\rm ext} - \Delta_\x u^{\rm ext}=0$ on $Q^{\rm ext}$. 
In particular, each component of $\nabla_\x u^{\rm ext}$ satisfies the homogeneous heat equation strongly on $Q^{\rm ext}$. 
We may thus apply the mean-value property for the heat equation, e.g.,~\cite[Section~2.3, Theorem~3]{evans10} for each component:
Let $R>0$ be sufficiently large such that $\overline\Omega \subset B_R(\boldsymbol{0})$ and $|u^{\rm ext}(t,\x)| \le 1$ for all $t\in I$ and all $\x\in \Omega^{\rm ext}$ with $|\x|\ge R$. 
Now, we fix coordinates $t\in I$ and $\x\in \Omega^{\rm ext}$ with $|\x|\ge 2R$.
For $r>0$, we define the corresponding heat ball
\begin{align}\label{eq:heat_ball}
\begin{split}
	E(t,\x;r) &:= \set{(s,\y) \in \R^{d+1}}{G(t-s,\x-\y) \ge r^{-d}} 
	\\ 
	&= \set{(s,\y) \in \R^{d+1}}{t-\frac{r^2}{4\pi} \le s\le t, |\x-\y|^2 \le 2d (t-s) \ln \frac{r^2}{4\pi (t-s)}}.
\end{split}
\end{align}
For sufficiently small $r>0$ with $E(t,\x;r)\subset Q^{\rm ext}$, the following mean-value property holds:
\begin{align}\label{eq:mean_value}
	\nabla_\x u^{\rm ext}(t,\x) = \frac{1}{4r^d} \int_{E(t,\x;r)} \nabla_\x u^{\rm ext}(s,\y) \frac{|\x-\y|^2}{(t-s)^2} \d\y \d s.
\end{align}
Note that for the choice $r=\sqrt{t}$, the resulting term $2d (t-s) \ln \frac{t}{4\pi (t-s)}$ of~\eqref{eq:heat_ball} is uniformly bounded by a generic constant for $t\in I$ and $\tfrac{3}{4\pi} t \le  s \le t$. 
Hence, we may assume that $R^2$ is also larger than this upper bound to guarantee that the closed ball $B(\x;t,s) := \set{\y\in\R^d}{|\x-\y|^2 \le 2d (t-s) \ln \frac{t}{4\pi (t-s)}}$ is contained in $\Omega^{\rm ext}\setminus B_R(\boldsymbol{0})$,
 and thus $E(t,\x;\sqrt{t}) \subset Q^{\rm ext}$. 
The mean-value property~\eqref{eq:mean_value} and integration by parts hence yield that 
\begin{align*}
	&\nabla_\x u^{\rm ext}(t,\x) = \frac{1}{4t^{d/2}} \int_{\frac{3}{4\pi} t}^t \int_{B(\x;t,s)} \nabla_\y u^{\rm ext}(s,\y) \frac{|\x-\y|^2}{(t-s)^2} \d\y \d s
	\\
	&= \frac{1}{4t^{d/2}} \Big(\int_{\frac{3}{4\pi} t}^t \int_{B(\x;t,s)} u^{\rm ext}(s,\y) \frac{2(\x-\y)}{(t-s)^2} \d\y \d s
		+ \int_{\frac{3}{4\pi} t}^t \int_{\partial B(\x;t,s)} u^{\rm ext}(s,\y) \n_{B(\x;t,s)}(\y) \frac{|\x-\y|^2}{(t-s)^2} \d\y \d s\Big).
\end{align*}
Here, $\n_{B(\x;t,s)}(\y)$ abbreviates the outward normal vector corresponding to $B(\x;t,s)$. 
Within both integration domains, we can further exploit that $|u^{\rm ext}(s,\y)|\le 1$. 
Together with the use of polar coordinates, we see that 
\begin{align*}
	&|\nabla_\x u^{\rm ext}(t,\x)| \le \frac{1}{4t^{d/2}} \Big(\int_{\frac{3}{4\pi} t}^t \int_{B(\x;t,s)} \frac{2|\x-\y|}{(t-s)^2} \d\y \d s
		+ \int_{\frac{3}{4\pi} t}^t \int_{\partial B(\x;t,s)} \frac{|\x-\y|^2}{(t-s)^2} \d\y \d s\Big)
	\\
	&\lesssim \frac{1}{t^{d/2}} \int_{\frac{3}{4\pi} t}^t (t-s)^{(d-3)/2} \Big(\ln \frac{t}{4\pi (t-s)}\Big)^{(d+1)/2} \d s
	= \frac{1}{t^{d/2}} \int_{0}^{\frac{1}{4\pi}t} s^{(d-3)/2} \Big(\ln \frac{t}{4\pi s}\Big)^{(d+1)/2} \d s.
\end{align*}
If $d=3$, the last integral is just $\tfrac{1}{2\pi}t$, which concludes the proof in this case even for $\epsilon=0$. 
If $d=2$, we assume without loss of generality that $0<\epsilon<1/2$.
We write $\ln \frac{t}{4\pi s} = \ln t - \ln  4\pi - \ln s \le \ln T - \ln s \le 2 \max\{|\ln T|, |\ln s|\}$ and get that 
\begin{align*}
	|\nabla_\x u^{\rm ext}(t,\x)|
	&\lesssim \frac{1}{t} \int_{0}^{\frac{1}{4\pi}t} s^{-1/2} \Big(|\ln T| ^{3/2} + |\ln s|^{3/2}\Big)\d s
	\\
	&= \frac{1}{t} \int_{0}^{\frac{1}{4\pi}t} s^{-1/2-\epsilon} \, s^\epsilon \Big(|\ln T| ^{3/2} + |\ln s|^{3/2}\Big)\d s
	\\
	&\le \frac{\max_{s\in[0,T]}s^\epsilon \big(|\ln T| ^{3/2} + |\ln s|^{3/2}\big)}{t} \int_{0}^{\frac{1}{4\pi}t} s^{-1/2-\epsilon} \d s
	\lesssim t^{-1/2-\epsilon}.
\end{align*}
This concludes the proof.
\end{proof}

\begin{proposition}
Let $u^{\rm ext}$ be a weak solution of the exterior problem~\eqref{eq:heat_exterior}\&\eqref{eq:initial_exterior} with radiation condition~\eqref{eq:radiation}. 
Then, it satisfies the representation formula
\begin{align}\label{eq:exterior_solution}
	u^{\rm ext} 
	= \widetilde{\mathcal{K}} (u^{\rm ext}|_\Sigma) - \widetilde{\mathcal{V}} (\nabla_\x u^{\rm ext}\cdot\n_\x)
	\quad \text{on } Q^{\rm ext}.
\end{align}
\end{proposition}

\begin{proof}
It remains to show that the last two terms in~\eqref{eq:exterior_solution_R} tend to zero.
To this end, let $(t,\x)\in Q_R^{\rm ext}$ be arbitrary but fixed. 
The simple estimate $|G(s,\y)| \lesssim s^{1/2} |\y|^{-d-1}$ for all $s\ge0$, $\y\in\R^d$ (see, e.g., the proof of \cite[Theorem~3.11]{costabel90}), the preceding lemma, and $|\Gamma_R| \simeq R^{d-1}$ show for sufficiently large $R>0$ and $0<\epsilon<1/2$ that
\begin{align*}
	\big|\big(\widetilde{\mathcal{V}}_R (\nabla_\x u^{\rm ext}\cdot\n_{R,\x})\big)(t,\x)\big|
	&= \Big|\int_{\Sigma_R} G(t-s,\x-\y) (\nabla_\y u^{\rm ext}(s,\y)\cdot\y\,|\y|^{-1}) \d\y \d s\Big|
	\\
	&\lesssim \int_0^t\int_{\Gamma_R} (t-s)^{1/2} |\x-\y|^{-d-1} s^{-1/2-\epsilon}\d\y \d s
	\\
	&\lesssim R^{-2} \int_0^t (t-s)^{1/2} s^{-1/2-\epsilon}\d\y \d s 
	\lesssim R^{-2} \to 0 \quad\text{as }R\to\infty.
\end{align*}
Using the radiation condition~\eqref{eq:radiation} instead of Lemma~\ref{lem:radiation}, we similarly see that
\begin{align*}
	\big|\big(\widetilde{\mathcal{K}}_R (u^{\rm ext}|_{\Sigma_R})\big)(t,\x) \big|
	&= \Big|\int_{\Sigma_R} \nabla_\y G(t-s,\x-\y)\cdot \y\,|\y|^{-1} \, u^{\rm ext}(s,\y) \d\y \d s\Big|
	\\
	&\lesssim \int_0^t \int_{\Gamma_R} \frac{|(\x-\y)\cdot\y|}{(t-s)^{d/2+1}|\y|} \, e^{-\frac{|\x-\y|^2}{4(t-s)}} \d\y \d s
	\\
	&\lesssim \int_0^t \int_{\Gamma_R} (t-s)^{-1/2} |\x-\y|^{-d} \, \d\y \d s
	\\
	&\lesssim R^{-1} \int_0^t (t-s)^{-1/2} \d s 
	\lesssim R^{-1} \to 0 \quad\text{as }R\to\infty.
\end{align*}
This concludes the proof.
\end{proof}

By the preceding proposition, once the weak solution $u$ of \eqref{eq:transmission} is known,
owing to the jump conditions~\eqref{eq:jump_dirichlet} and \eqref{eq:jump_neumann},
the corresponding $u^{\rm ext}$ can be computed as 
\begin{align}\label{eq:interior2exterior}
	u^{\rm ext} 
	= \widetilde{\mathcal{K}} (u|_\Sigma - u_D) - \widetilde{\mathcal{V}} ((\A\nabla_\x u) \cdot\n_\x - \phi_N). 
\end{align}
Conversely, in the proof of \cite[Theorem~3.11]{costabel90}, it is shown that any function of the form
\begin{align}\label{eq:representation_radiation}
	u^{\rm ext} 
	= \widetilde{\mathcal{K}} u - \widetilde{\mathcal{V}} \phi 
	\text{ satisfies }\eqref{eq:radiation} \text{ for all }u\in H^{1/2,1/4}(\Sigma), \phi\in H^{-1/2,-1/4}(\Sigma).
\end{align}

\begin{remark}\label{rem:radiation}
To be more precise, the cited \cite[Theorem~3.11]{costabel90} states for $u^{\rm ext}$
defined as in~\eqref{eq:representation_radiation} that $u^{\rm ext} (\cdot,\x)=\OO(|x|^{-d})$
and $\nabla_\x u^{\rm ext} (\cdot,\x) \cdot \x = \OO(|x|^{-d})$ as $|x|\to\infty$. 
Using the same argument, one even sees that $u^{\rm ext} (\cdot,\x)=\OO(|x|^{-\mu})$ and $\nabla_\x u^{\rm ext} (\cdot,\x) = \OO(|x|^{-\mu})$ as $|x|\to\infty$ for all fixed $\mu\ge0$ (with hidden constants depending only on $d$, $\mu$, $u^{\rm ext}$, and $T$).
The assumption~\eqref{eq:radiation} that $u^{\rm ext}$ converges to zero as $|\x|\to\infty$ thus automatically yields a much stronger decay of $u^{\rm ext}$ and its spatial gradient $\nabla_\x u^{\rm ext}$. 
On the other hand, to guarantee that the last two terms in \eqref{eq:exterior_solution_R} converge to zero, it is clear from our proof  that one could even allow for a certain growth of $u^{\rm ext} $ (and $\nabla_\x u^{\rm ext}$). 
Here and in the following, the \emph{minimal} required radiation condition would be to explicitly assume that those two terms vanish as $R\to\infty$. 
\end{remark}

\section{Coupling of FOSLS and boundary integral equations}

\subsection{First-order system least squares formulation for interior problem}\label{sec:fosls}
We define the spaces
\begin{align*}
	U &:= \boldsymbol{H}(\div;Q) \cap (L^2(I;H^1(\Omega))\times L^2(Q)^d),
	\\
	U_D &:= \set{(u,\ssigma)\in U}{u|_\Sigma=0},
	\\
	U_N &:= \set{(u,\ssigma)\in U}{\ssigma\cdot\n_\x=0},
	\\
	L &:= L^2(Q) \times L^2(Q)^d \times L^2(\Omega).
\end{align*}
Here, $\div$ denotes the time-space divergence  $\div(u,\ssigma)=\partial_t u + \div_\x \ssigma$ and
$\boldsymbol{H}(\div;Q) = \set{(u,\ssigma)\in L^2(Q)^{1+d}}{\div (u,\ssigma) \in L^2(Q)}$.
The spaces $U, U_D, U_N$ are equipped with the canonical graph norm
\begin{align*}
	\norm{(u,\ssigma)}{U}^2 := \norm{u}{L^2(Q)}^2 + \norm{\nabla_\x u}{L^2(Q)}^2 + \norm{\ssigma}{L^2(Q)}^2 + \norm{\div (u,\ssigma)}{L^2(Q)}^2. 
\end{align*}
We further define the bounded linear operators $\mathcal{G}: U\to L$, $\mathcal{G}_D: U\to L\times H^{1/2,1/4}(\Sigma)$, and $\mathcal{\mathcal{G}}_N: U\to L\times H^{-1/2,-1/4}(\Sigma)$  by
\begin{align*}
	\mathcal{G}(u,\ssigma) &:= \big(\div \u, \A\nabla_\x u  + \ssigma, u(0,\cdot)\big), 
	\\
	\mathcal{G}_D(u,\ssigma) &:= \big(\mathcal{G}(u,\ssigma),u|_\Sigma\big), 
	\\
	\mathcal{G}_N(u,\ssigma) &:= \big(\mathcal{G}(u,\ssigma),\ssigma\cdot \n_\x\big). 
\end{align*}
Note that $\u\cdot \n = \ssigma\cdot\n_\x$ for all $\u=(u,\ssigma)\in U$.
By \cite[Lemma~2.2]{gs21},
\begin{align}\label{eq:lemma_gs21}
  (u,\ssigma)\mapsto u \text{ is a bounded linear mapping from $U$ to $L^2(I;H^1(\Omega)) \cap H^1(I; H^{-1}(\Omega))$},
\end{align}
and~\eqref{eq:V2C0} implies that $\mathcal{G}$ is indeed well defined and bounded. 
According to \cite[Theorem~2.8]{gs21}, the latter two operators, $\mathcal{G}_D$ and $\mathcal{\mathcal{G}}_N$, are even bounded linear isomorphisms. 
Hence, $\mathcal{G}|_{U_D}$ and $\mathcal{G}|_{U_N}$ are bounded linear isomorphisms as well. 
%
%

We finally note that any weak solution $u$ of the interior problem~\eqref{eq:heat_interior}\&\eqref{eq:initial_interior} solves 
\begin{align}\label{eq:transmission_interior}
	\mathcal{G}\u = (f,\boldsymbol{0},u_0) \quad\text{with }\u=(u,\ssigma)=(u,-\A\nabla_\x u)\in U,
\end{align}
or equivalently, for arbitrary fixed $\alpha>0$,
\begin{align}\label{eq:transmission_interior2}
	b_\alpha(\u,\v)
	:=\alpha\dual{\mathcal{G}\u}{\mathcal{G}\v}_L 
	= \alpha \dual{(f,\boldsymbol{0},u_0)}{\mathcal{G}\v}_L \quad \text{for all }\v = (v,\ttau)\in U, 
\end{align} 
Here, $\dual{\cdot}{\cdot}_L$ denotes the scalar product on $L$.
More explicitly, this reads as 
\begin{align*}
	b_\alpha(\u,\v)
	&=\alpha \big[\dual{\div \u}{\div \v}_{L^2(Q)} + \dual{\A\nabla_\x u + \ssigma}{\A\nabla_\x v  + \ttau}_{L^2(Q)} + \dual{u(0,\cdot)}{v(0,\cdot)}_{L^2(\Omega)}\big]
	\\
	&= \alpha \big[\dual{f}{\div \v}_{L^2(Q)} + \dual{u_0}{v(0,\cdot)}_{L^2(\Omega)}\big],
\end{align*} 
where $\dual{\cdot}{\cdot}_{L^2(Q)}$ denotes the scalar product on $L^2(Q)$ and $\dual{\cdot}{\cdot}_{L^2(\Omega)}$ the scalar product on $L^2(\Omega)$.

\subsection{Coupling to exterior problem via boundary integral equations}\label{sec:couplings}
We introduce couplings to three different boundary integral equations obtained from the representation formula~\eqref{eq:interior2exterior} for $u^{\rm ext}$.
While the first two couplings only involve two of the four boundary integral operators but are
only provably coercive in case of $\A={\bf Id}$, the third coupling involves all four boundary
integral operators but is coercive for any symmetric uniformly positive definite
$\A\in L^\infty(Q)^{d\times d}$.

\subsubsection{Dirichlet trace of representation formula}\label{sec:dirichlet_coupling}
If $u^{\rm ext}$ is the corresponding weak solution of the transmission problem~\eqref{eq:transmission},
its representation~\eqref{eq:interior2exterior} restricted to $\Sigma$
(using the identities~\eqref{eq:single_operator}--\eqref{eq:double_operator}),
the jump conditions~\eqref{eq:jump_dirichlet}, and the fact that
$(\A\nabla_\x u)\cdot\n_\x = -\u\cdot \n = -\ssigma\cdot\n_\x$ from~\eqref{eq:transmission_interior}
yield that
\begin{align}\label{eq:transmission_exterior}
	\mathcal{V} (-\ssigma \cdot \n_\x) + (\tfrac12-\mathcal{K}) (u|_\Sigma) = \mathcal{V} \phi_N + (\tfrac12-\mathcal{K}) u_D.
\end{align}
Since $\mathcal{G}_N(U) = L\times H^{-1/2,-1/4}(\Sigma)$, this is equivalent to the following variational formulation: 
\begin{align*}
	c_1(\u,\v):=\dual{\mathcal{V} (\ssigma \cdot \n_\x) + (\mathcal{K}-\tfrac12) (u|_\Sigma)}{\ttau \cdot \n_\x}_\Sigma 
	= \dual{-\mathcal{V} \phi_N + (\mathcal{K}-\tfrac12) u_D}{\ttau \cdot \n_\x}_\Sigma
	\\
	\text{ for all }\v=(v,\ttau)\in U.
\end{align*}
Combining~\eqref{eq:transmission_interior2} and~\eqref{eq:transmission_exterior}, we conclude that any solution of~\eqref{eq:transmission} with $\u=(u,-\A\nabla_\x u)$ satisfies
\begin{align} \notag
	&b_\alpha(\u,\v) + c_1(\u,\v)  
	= \alpha \big[\dual{f}{\div \v}_{L^2(Q)} + \dual{u_0}{v(0,\cdot)}_{L^2(\Omega)}\big] + \dual{-\mathcal{V} \phi_N + (\mathcal{K}-\tfrac12) u_D}{\ttau\cdot\n_\x}_\Sigma
	\\ \label{eq:transmission_weak}&\hspace{11cm}
	\text{for all }\v = (v,\ttau)\in U.
\end{align}
By considering $\v\in U_N$ and recalling that $\mathcal{G}(U_N)=L$, one sees that \eqref{eq:transmission_weak}
conversely implies \eqref{eq:transmission_interior} and thus also \eqref{eq:transmission_exterior}.
Note that the left-hand side and the right-hand side in \eqref{eq:transmission_weak} are linear and
continuous with respect to inputs $\u,\v\in U$. 

\subsubsection{Neumann trace of representation formula}\label{sec:neumann_coupling}
Alternatively, one can consider the normal derivative of \eqref{eq:interior2exterior} and use
the identities~\eqref{eq:adjdouble_operator}--\eqref{eq:hypsing_operator}.
Then, the jump condition~\eqref{eq:jump_neumann} and
$(\A\nabla_\x u)\cdot\n_\x = -\ssigma\cdot\n_\x$ from~\eqref{eq:transmission_interior} yield the equation
\begin{align}\label{eq:transmission_exterior2}
	\mathcal{W}(u|_\Sigma) + (\mathcal{N} + \tfrac12)(-\ssigma\cdot\n_\x) 
	= \mathcal{W}u_D + (\mathcal{N} + \tfrac12)\phi_N,
\end{align}
or equivalently, owing to $\mathcal{G}_D(U) = L\times H^{1/2,1/4}(\Sigma)$,
\begin{align*}
	c_2(\u,\v):=\dual{\mathcal{W}(u|_\Sigma) + (\mathcal{N} + \tfrac12)(-\ssigma\cdot\n_\x)}{v|_\Sigma}_\Sigma 
	= \dual{\mathcal{W}u_D + (\mathcal{N} + \tfrac12)\phi_N}{v|_\Sigma}_\Sigma 
	\\
	\text{for all }\v=(v,\ttau)\in U.
\end{align*}
Combining~\eqref{eq:transmission_interior2} and~\eqref{eq:transmission_exterior}, we conclude that any solution of~\eqref{eq:transmission} with $\u=(u,-\A\nabla_\x u)$ satisfies
\begin{align} \notag
	&b_\alpha(\u,\v) + c_2(\u,\v) 
	= \alpha \big[\dual{f}{\div \v}_{L^2(Q)} + \dual{u_0}{v(0,\cdot)}_{L^2(\Omega)}\big]
	+ \dual{\mathcal{W}u_D + (\mathcal{N} + \tfrac12)\phi_N}{v|_\Sigma}_\Sigma 
	\\ \label{eq:transmission_weak2} & \hspace{11cm}
	\text{for all }\v = (v,\ttau)\in U.
\end{align}
By considering $\v\in U_D$ and recalling that $\mathcal{G}(U_D)=L$, one sees that \eqref{eq:transmission_weak2} conversely implies \eqref{eq:transmission_interior} and thus also \eqref{eq:transmission_exterior2}.
Note that the left-hand side and the right-hand side in \eqref{eq:transmission_weak2} are linear and continuous with respect to inputs $\u,\v\in U$.

\subsubsection{Dirichlet and Neumann trace of representation formula}
Finally, one can also consider the sum of the two bilinear forms $c_1$ and $c_2$,
which then corresponds to the sum of the two variational formulations, i.e., 
\begin{align*}
	c_3(\u,\v) := \dual{\mathcal{V} (\ssigma \cdot \n_\x) + (\mathcal{K} - \tfrac12 ) (u|_\Sigma)}{\ttau \cdot \n_\x}_\Sigma + \dual{\mathcal{W}(u|_\Sigma) + (\mathcal{N}+\tfrac12)(-\ssigma\cdot\n_\x)}{v|_\Sigma}_\Sigma \\
	= \dual{-\mathcal{V} \phi_N + (\mathcal{K} - \tfrac12) u_D}{\ttau \cdot \n_\x}_\Sigma 
	+ \dual{\mathcal{W}u_D + (\mathcal{N} + \tfrac12)\phi_N}{v|_\Sigma}_\Sigma 
	\quad\text{for all }\v=(v,\ttau)\in U.
\end{align*}
We overall conclude from \eqref{eq:transmission} that
\begin{align} \notag
	&b_\alpha(\u,\v) + c_3(\u,\v) 
	= \alpha \big[\dual{f}{\div \v}_{L^2(Q)} + \dual{u_0}{v(0,\cdot)}_{L^2(\Omega)}\big]
	+ \dual{-\mathcal{V} \phi_N + (\mathcal{K} - \tfrac12) u_D}{\ttau \cdot \n_\x}_\Sigma 
	\\ \label{eq:transmission_weak3} &\hspace{5cm}
	+ \dual{\mathcal{W}u_D + (\mathcal{N} + \tfrac12)\phi_N}{v|_\Sigma}_\Sigma 
	\quad\text{for all }\v = (v,\ttau)\in U.
\end{align}
Contrary to Sections~\ref{sec:dirichlet_coupling} and~\ref{sec:neumann_coupling}, it is not evident that \eqref{eq:transmission_weak3} conversely implies \eqref{eq:transmission_interior} and/or either of the equations \eqref{eq:transmission_exterior} and \eqref{eq:transmission_exterior2}.
However, in Section~\ref{sec:equivalence}, we will see that this is indeed the case. 
Note that the left-hand side and the right-hand side in \eqref{eq:transmission_weak3} are linear and continuous with respect to inputs $\u,\v\in U$.

\subsection{Coercivity of combined bilinear forms}\label{sec:coercivity}
Our next goal is to show that the continuous bilinear forms
$b_\alpha + c_i: U \times U \to \R$ are coercive for $i\in\{1,2,3\}$ and sufficiently large $\alpha>0$.
We start with the following auxiliary result.

\begin{lemma}\label{lem:Aubin-Lions-Simon}
It holds that 
\begin{align*}
	\dual{u|_\Sigma}{-\ssigma\cdot\n_\x}_\Sigma  
	= \norm{\A^{1/2}\nabla_\x u}{L^2(Q)}^2 + \tfrac12\norm{u(T,\cdot)}{L^2(\Omega)}^2
	\gtrsim \norm{\u}{U}^2
	\quad \text{for all } \u=(u,\ssigma)\in{\rm ker}\,\mathcal{G}.
\end{align*}
\end{lemma}

\begin{proof}
The kernel of $\mathcal{G}$ is given by the set 
\begin{align*}
	{\rm ker}\,\mathcal{G} = \set{(u,\ssigma)\in U}{\partial_t u - \div_\x (\A\nabla_\x u) = 0 \wedge \ssigma = -\A\nabla_\x u \wedge u(0,\cdot) = 0}.
\end{align*}
Let $\u = (u,\ssigma)\in {\rm ker}\,\mathcal{G}$. 
Then, boundedness~\eqref{eq:lemma_gs21} shows that
\begin{align}\label{eq:U2V}
	\norm{u}{L^2(I; H^1(\Omega))\cap H^1(I;H^{-1}(\Omega))}^2 \eqsim \norm{\u}{U}^2 
	\eqsim \norm{u}{L^2(Q)}^2 + \norm{\nabla_\x u}{L^2(Q)}^2.
\end{align}
Note that $(\A\nabla_\x u) \cdot \n_\x = -\ssigma\cdot\n_\x$.
Integration by parts similar to~\cite[Equations~(2.35)\&(2.52)]{costabel90} together with $\partial_t u - \div_\x (\A\nabla_\x u) = 0$ and $u(0,\cdot)=0$ gives that
\begin{align*}
	\dual{u|_\Sigma}{-\ssigma\cdot\n_\x}_\Sigma  
	&=\int_Q (\A\nabla_\x u)\cdot \nabla_\x u \d\x \d t - \int_Q \big(\partial_t u - \div_\x(\A\nabla_\x u)\big)\,  u \d\x \d t +\int_Q \partial_t u  \,u \d\x \d t
	\\
	&= \norm{\A^{1/2}\nabla_\x u}{L^2(Q)}^2 + \tfrac12\norm{u(T,\cdot)}{L^2(\Omega)}^2. 
\end{align*}

Next, we want to apply a standard compactness argument, more precisely, the Peetre--Tartar lemma~\cite[Lemma~A.20]{eg21a}. 
In the notation from there (with $S$ instead of $T$, being already the end time point), we set 
\begin{align*}
	X&:=\set{v\in L^2(I;H^1(\Omega))\cap H^1(I;H^{-1}(\Omega))}{\partial_t v - \div_\x(\A\nabla_\x v) = 0},
	\\ 
	Y&:=L^2(Q)\times L^2(\Omega),
	\\
	Z&:=L^2(Q),
\end{align*}
and 
\begin{align*}
	&A:X\to Y, \quad v\mapsto \big(\A^{1/2}\nabla_\x v, \tfrac{1}{\sqrt2} v(T,\cdot)\big),
	\\
	&S:X\to Z, \quad v\mapsto v.
\end{align*}
Again by~\eqref{eq:lemma_gs21}, it holds that 
\begin{align*}
	\norm{v}{X} \lesssim \norm{v}{L^2(Q)} + \norm{\A^{1/2}\nabla_\x v}{L^2(Q)} \le \norm{Av}{Y} + \norm{Sv}{Z} \quad\text{for all }v\in X.
\end{align*}
Moreover, the operator $A$ is bounded owing to the continuous embedding~\eqref{eq:V2C0}.
It is also injective:
Indeed, if $Av=(0,0)$, then $\nabla_\x v = 0$ and thus $\partial_t v = \div_\x \A \nabla_\x v = 0$ so that $v$ is constant.
As $v(T,\cdot)=0$, we conclude that $v=0$. 
Moreover, the operator $S$ is compact. 
This follows from the continuous inclusion $X\hookrightarrow L^2(I;H^1(\Omega))\cap H^1(I;H^{-1}(\Omega))$ and the Aubin--Lions--Simons lemma~\cite[Theorem~II.5.16]{bf12} stating that the inclusion $L^2(I;H^1(\Omega))\cap H^1(I;H^{-1}(\Omega))$ $\hookrightarrow L^2(Q)$ is compact. 
Overall, the Peetre--Tartar lemma yields that
\begin{align*}
	\norm{v}{X} \lesssim \norm{Av}{Y} \quad\text{for all }v\in X.
\end{align*}
Since $\norm{Au}{Y}^2 = \norm{\A^{1/2}\nabla_\x u}{L^2(Q)}^2 + \tfrac12\norm{u(T,\cdot)}{L^2(\Omega)}^2$, we conclude the proof with \eqref{eq:U2V}.
\end{proof}

For the proof of coercivity of the bilinear forms $c_1$ and $c_2$, we restrict to the case $\A={\bf Id}$.

\begin{lemma}\label{lem:coercivity on kernel}
Suppose that $\A = {\bf Id}$.
Then, the continuous bilinear forms $c_1:U\times U\to \R$ and $c_2:U\times U\to\R$ are coercive on the kernel of $\mathcal{G}$, i.e., 
\begin{align*}
	c_i(\u,\u) \gtrsim \norm{\u}{U}^2
	\quad\text{for }i\in\{1,2\} \text{ and all }\u\in {\rm ker}\,\mathcal{G}.
\end{align*}
\end{lemma}

\begin{proof}
Let $\u = (u,\ssigma)\in {\rm ker}\,\mathcal{G}$. 
Since $\partial_t u - \Delta_\x u = 0$ with $u(0,\cdot) = 0$, the representation formula~\eqref{eq:representation} restricted to $\Sigma$ and the identities~\eqref{eq:single_operator}--\eqref{eq:double_operator} yield that 
\begin{align*}
	\mathcal{V}(\nabla_\x u \cdot \n_\x) + (\tfrac12-\mathcal{K})(u|_\Sigma) = u|_\Sigma.
\end{align*}
Note that $\nabla_\x u \cdot \n_\x = -\ssigma\cdot\n_\x$.
This yields that $c_1(\u,\u) = \dual{u|_\Sigma}{-\ssigma\cdot\n_\x}_\Sigma$. 
Similarly, taking the normal derivative of the representation formula~\eqref{eq:representation} and the identities~\eqref{eq:adjdouble_operator}--\eqref{eq:hypsing_operator} yield that 
\begin{align*}
	\mathcal{W}(u|_\Sigma) + (\mathcal{N} + \tfrac12)(\nabla_x u\cdot\n_\x) 
	= \nabla_x u\cdot\n_\x,
\end{align*}
and thus, $c_2(\u,\u) = \dual{u|_\Sigma}{-\ssigma\cdot\n_\x}_\Sigma$. 
In both cases, Lemma~\ref{lem:Aubin-Lions-Simon} concludes the proof.
\end{proof}

Obviously, the bilinear form $c_3$, being the sum of the two bilinear forms $c_1$ and $c_2$ is also coercive on the kernel of $\mathcal{G}$, provided that $\A = {\bf Id}$. 
Using a different proof, it turns out that the bilinear form $c_3$ is coercive on the kernel of $\mathcal{G}$ even without any restrictions on the diffusion matrix $\A$ in the interior space-time cylinder~$Q$.

\begin{lemma}\label{lem:coercivity on kernel 3}
The continuous bilinear form $c_3:U\times U\to \R$ is coercive on the kernel of $\mathcal{G}$, i.e., 
\begin{align*}
	c_3(\u,\u) \gtrsim \norm{\u}{U}^2
	\quad\text{for all }\u\in {\rm ker}\,\mathcal{G}.
\end{align*}
\end{lemma}

\begin{proof}
The coercivity~\eqref{eq:Calderon_coercive} yields the existence of some constant $C>0$ such that for \emph{any} $\u = (u,\ssigma) \in U$,
\begin{align*}
	c_3(\u,\u) &= \dual{\mathcal{V} (\ssigma \cdot \n_\x) + (\mathcal{K} - \tfrac12 ) (u|_\Sigma)}{\ssigma \cdot \n_\x}_\Sigma + \dual{\mathcal{W}(u|_\Sigma) + (\mathcal{N}+\tfrac12)(-\ssigma\cdot\n_\x)}{u|_\Sigma}_\Sigma \\
	&=(-\ssigma\cdot\n_\x,u|_\Sigma) 
	\begin{pmatrix}\mathcal{V} & -\mathcal{K} \\ \mathcal{N} & \mathcal{W} \end{pmatrix} 
	\begin{pmatrix}-\ssigma\cdot\n_\x \\ u|_\Sigma\end{pmatrix}
	+ \dual{-\ssigma\cdot\n_\x}{u|_\Sigma}_\Sigma
	\\
	&\ge C \big(\norm{\ssigma\cdot\n_\x}{H^{-1/2,-1/4}(\Sigma)}^2 + \norm{u|_\Sigma}{H^{1/2,1/4}(\Sigma)}^2\big)
	+ \dual{-\ssigma\cdot\n_\x}{u|_\Sigma}_\Sigma
	\\
	&\ge \dual{-\ssigma\cdot\n_\x}{u|_\Sigma}_\Sigma.
\end{align*}
Lemma~\ref{lem:Aubin-Lions-Simon} concludes the proof.
\end{proof}

With Lemma~\ref{lem:coercivity on kernel} and Lemma~\ref{lem:coercivity on kernel 3},
coercivity of $b_\alpha+c_i$, $i\in\{1,2,3\}$, follows from the following abstract theorem, which is also implicitly used in~\cite{fhk17}.

\begin{theorem}\label{thm:coercivity}
Let $U$ and $L$ be Hilbert spaces and $\mathcal{G}:U\to L$ a continuous linear mapping such that $\mathcal{G}|_{U_D}$ is even a continuous linear isomorphism onto $L$ for some closed subspace $U_D\subseteq U$. 
For $\alpha>0$, define the continuous bilinear form $b_\alpha:U\times U$ by $b_\alpha(\u,\v) := \alpha\dual{\mathcal{G}\u}{\mathcal{G}\v}_L$. 
Moreover, let $c:U\times U$ be a second continuous bilinear form that is coercive on ${\rm ker}\,\mathcal{G}$. 
Then, there exists $\alpha_0>0$ such that for all $\alpha\ge \alpha_0$, $b_\alpha+c$ is coercive on $U$, i.e.,
\begin{align*}
	b_\alpha(\u,\u) + c(\u,\u) \gtrsim \norm{\u}{U}^2 \quad\text{for all }\u\in U.
\end{align*}
\end{theorem}

\begin{proof}
The operator $\mathcal{G}$ induces the decomposition $U=U_D+{\rm ker}\,\mathcal{G}$:
Indeed, each $\u\in U$ can be uniquely written as $\u = \mathcal{G}|_{U_D}^{-1}\mathcal{G}\u + (\u - \mathcal{G}|_{U_D}^{-1}\mathcal{G}\u)$. 
We abbreviate the corresponding projection with range $U_D$ and kernel ${\rm ker}\,\mathcal{G}$ by $\mathcal{P}:=\mathcal{G}|_{U_D}^{-1}\mathcal{G}$. 
Then, the assertion follows from the assumptions and the elementary calculations 
\begin{align*}
	\norm{\u}{U}^2 &\lesssim \norm{\mathcal{P}\u}{U}^2 + \norm{\u-\mathcal{P}\u}{U}^2
	\\
	&\lesssim \norm{\mathcal{G} \u}{L}^2 + c(\u-\mathcal{P}\u,\u-\mathcal{P}\u)
	\\
	&= \norm{\mathcal{G} \u}{L}^2 + c(\u,\u) - c(\u,\mathcal{P}\u) - c(\mathcal{P}\u,\u) + c(\mathcal{P}\u,\mathcal{P}\u)
	\\
	&\lesssim \norm{\mathcal{G} \u}{L}^2 + c(\u,\u) + \norm{\u}{U}\norm{\mathcal{P}\u}{U} + \norm{\mathcal{P}\u}{U}\norm{\mathcal{P}\u}{U}
	\\
	&\lesssim \norm{\mathcal{G} \u}{L}^2 + c(\u,\u) + \norm{\u}{U}\norm{\mathcal{G}\u}{L}
	\\
	&\le (1+\tfrac{\delta^{-1}}{2})\norm{\mathcal{G} \u}{L}^2 + c(\u,\u) + \tfrac{\delta}{2}\norm{\u}{U}^2.
\end{align*}
Choosing $\delta>0$ sufficiently small, this proves coercivity for suitably chosen $\alpha_0$.
\end{proof}

\begin{remark}
Obviously, coercivity of the bilinear form $c$ on the kernel of $\mathcal{G}$ is also necessary for coercivity of the bilinear form $b_\alpha+c$. 
\end{remark}

\subsection{Equivalence of couplings to transmission problem}\label{sec:equivalence}
It is straight-forward to see that the first two couplings~\eqref{eq:transmission_weak} and \eqref{eq:transmission_weak2} are equivalent to the transmission problem~\eqref{eq:transmission}, even without any restrictions on the diffusion matrix $\A$; see Theorem~\ref{thm:equivalence} and Remark~\ref{rem:equivalence} below. 
This is not evident for the third coupling~\eqref{eq:transmission_weak3}, which involves all four boundary integral operators. 
However, we have already seen that \eqref{eq:transmission} implies each of the couplings~\eqref{eq:transmission_weak}, \eqref{eq:transmission_weak2}, and \eqref{eq:transmission_weak3}, and that the latter are uniquely solvable provided that $\alpha$ is chosen sufficiently large (see Theorem~\ref{thm:coercivity}) and $\A={\bf Id}$ in case of \eqref{eq:transmission_weak} and \eqref{eq:transmission_weak2}. 
If we knew that \eqref{eq:transmission} has indeed a unique weak solution, the desired equivalence would thus follow at once. 
While the problem at hand has been considered before, we were unable to find a rigorous proof of unique solvability. 
Therefore, we address this seemingly open question in the following.

We introduce yet another (though practically unfeasible) coupling, starting from the first-order system~\eqref{eq:transmission_interior} for $\u=(u,\ssigma)\in U$ in the interior and the boundary integral equation~\eqref{eq:transmission_exterior}, i.e., the representation~\eqref{eq:interior2exterior} of the exterior solution $u^{\rm ext}$ restricted to $\Sigma$. 
Note that~\eqref{eq:transmission_exterior} is equivalent to 
\begin{align}\label{eq:transmission_exterior4}
	-\ssigma \cdot \n_\x + \mathcal{V}^{-1}(\tfrac12-\mathcal{K}) (u|_\Sigma) 
	= \phi_N + \mathcal{V}^{-1}(\tfrac12-\mathcal{K}) u_D,
\end{align}
or, owing to $\mathcal{G}_D(U) = L\times H^{1/2,1/4}(\Sigma)$,
\begin{align*}
	c_4(\u,\v)
	:=\dual{-\ssigma \cdot \n_\x + \mathcal{V}^{-1}(\tfrac12-\mathcal{K}) (u|_\Sigma)}{v|_\Sigma}_\Sigma
	= \dual{\phi_N + \mathcal{V}^{-1}(\tfrac12-\mathcal{K}) u_D}{v|_\Sigma}_\Sigma
	\\
	\text{for all }\v=(v,\ttau)\in U.
\end{align*}
Combining~\eqref{eq:transmission_interior2} and~\eqref{eq:transmission_exterior}, we conclude that any solution of~\eqref{eq:transmission} with $\u=(u,-\A\nabla_\x u)$ satisfies
\begin{align} \notag
	& b_\alpha(\u,\v) + c_4(\u,\v) 
	= \alpha \big[\dual{f}{\div \v}_{L^2(Q)} + \dual{u_0}{v(0,\cdot)}_{L^2(\Omega)}\big]
	+ \dual{\phi_N + \mathcal{V}^{-1}(\tfrac12-\mathcal{K}) u_D}{v|_\Sigma}_\Sigma
	\\ \label{eq:transmission_weak4} &\hspace{11cm}
	\text{for all }\v = (v,\ttau)\in U.
\end{align}
By considering $\v\in U_D$ and recalling that $\mathcal{G}(U_D)=L$, one sees that \eqref{eq:transmission_weak4} conversely implies \eqref{eq:transmission_interior} and thus also \eqref{eq:transmission_exterior4} and the equivalent~\eqref{eq:transmission_exterior}.
Note that the left-hand side and the right-hand side in \eqref{eq:transmission_weak4} are linear and continuous with respect to inputs $\u,\v\in U$.

We want to show that the combined bilinear form $b_\alpha+c_4:U\times U\to\R$ is coercive for sufficiently large $\alpha>0$ and \emph{arbitrary} diffusion matrix $\A$ in the interior.
As for the other bilinear forms $b_\alpha+c_i$, $i\in\{1,2,3\}$, this directly follows from the following lemma; see Theorem~\ref{thm:coercivity}.

\begin{lemma}\label{lem:coercivity on kernel 4}
The continuous bilinear form $c_4:U\times U\to\R$ is coercive on the kernel of $\mathcal{G}$, i.e., 
\begin{align*}
	c_4(\u,\u) \gtrsim \norm{\u}{U}^2
	\quad\text{for all }\u\in {\rm ker}\,\mathcal{G}.
\end{align*}
\end{lemma}

\begin{proof}
Let $\u = (u,\ssigma)\in {\rm ker}\,\mathcal{G}$. 
Then, Lemma~\ref{lem:Aubin-Lions-Simon} yields the existence of a constant $C>0$ such that
\begin{align*}
	c_4(\u,\u)
	&=\dual{-\ssigma \cdot \n_\x + \mathcal{V}^{-1}(\tfrac12-\mathcal{K}) (u|_\Sigma)}{u|_\Sigma}_\Sigma
	\\
	&= \norm{\A^{1/2}\nabla_\x u}{L^2(Q)}^2 + \tfrac12\norm{u(T,\cdot)}{L^2(\Omega)}^2
	+ \dual{\mathcal{V}^{-1}(\tfrac12-\mathcal{K}) (u|_\Sigma)}{u|_\Sigma}_\Sigma
	\\
	&\ge C \norm{\u}{U}^2 
	+ \dual{\mathcal{V}^{-1}(\tfrac12-\mathcal{K}) (u|_\Sigma)}{u|_\Sigma}_\Sigma.
\end{align*}

It remains to show that the second term is nonnegative. 
To this end, we consider the function 
\begin{align*}
	w:=\widetilde{\mathcal{K}}(u|_\Sigma) - \widetilde{\mathcal{V}}\phi 
	\quad \text{for } 
	\phi:=\mathcal{V}^{-1}(\mathcal{K}-\tfrac12) (u|_\Sigma)
\end{align*}
on the exterior space-time cylinder $Q^{\rm ext}$. 
By restricting $w$ to $\Sigma$, we see that 
\begin{align*}
	w|_\Sigma = (\mathcal{K}+\tfrac12)(u|_\Sigma) - \mathcal{V}\phi = u|_\Sigma.
\end{align*}
On the other hand, from $\partial_t w - \Delta_\x w = 0$, $w(0,\cdot)=0$, and the fact that $w$ satisfies the radiation condition~\eqref{eq:radiation} (with $u^{\rm ext}$ replaced by $w$; see~\eqref{eq:representation_radiation}), we get as in~\eqref{eq:exterior_solution} the representation formula $w=\widetilde{\mathcal{K}}(w|_\Sigma) - \widetilde{\mathcal{V}}(\nabla_\x w\cdot\n_\x)$ and thus $\widetilde{\mathcal{V}}(\phi) = \widetilde{\mathcal{V}}(\nabla_\x w\cdot\n_\x)$ and finally $\phi = \nabla_\x w\cdot\n_\x$. 
Overall, this yields that 
\begin{align*}
	\dual{\mathcal{V}^{-1}(\tfrac12-\mathcal{K}) (u|_\Sigma)}{u|_\Sigma}_\Sigma 
	= \dual{-\phi}{u|_\Sigma}_\Sigma 
	= \dual{-\nabla_\x w\cdot\n_\x}{w|_\Sigma}_\Sigma.
\end{align*}
Lemma~\ref{lem:Aubin-Lions-Simon} applied to $\u = (w,-\nabla_\x w)$, $\A={\rm Id}$, and $Q_R^{\rm ext}$ instead of $Q$ gives that 
\begin{align*}
	\dual{-\nabla_\x w\cdot\n_\x}{w|_\Sigma}_\Sigma 
	= \norm{\nabla_\x w}{L^2(Q_R^{\rm ext})}^2 + \tfrac12\norm{w(T,\cdot)}{L^2(\Omega_R^{\rm ext})}^2 - \dual{\nabla_\x w\cdot\n_{R,\x}}{w|_{\Sigma_R}}_{\Sigma_R}.
\end{align*}
Elementary analysis and the stronger radiation condition $w(\x)=\OO(|x|^{-d})$ and $\nabla_\x w(\x)\cdot \x = \OO(|x|^{-d})$ as $|x|\to\infty$ from Remark~\ref{rem:radiation} show that
$\dual{\nabla_\x w\cdot\n_{R,\x}}{w|_{\Sigma_R}}_{\Sigma_R} \to 0$ as $R\to\infty$;
see also the proof of \cite[Theorem~3.11]{costabel90}.
The monotone convergence theorem yields that
\begin{align*}
	\dual{-\nabla_\x w\cdot\n_\x}{w|_\Sigma}_\Sigma 
	= \norm{\nabla_\x w}{L^2(Q^{\rm ext})}^2 + \tfrac12\norm{w(T,\cdot)}{L^2(\Omega^{\rm ext})}^2.
\end{align*} 
We conclude that $\dual{-\nabla_\x w\cdot\n_\x}{w|_\Sigma}_\Sigma$ is indeed nonnegative and thus the proof. 
\end{proof}

In the next theorem, we finally show that, conversely, the first component of the unique solution 
$\u = (u,\ssigma) \in U$ of the variational formulation~\eqref{eq:transmission_weak4}, 
which itself is equivalent to the first-order system~\eqref{eq:transmission_interior} and the 
boundary integral equation~\eqref{eq:transmission_exterior}, and the associated exterior function  
\begin{align}\label{eq:2:interior2exterior}
	u^{\rm ext} := \widetilde{\mathcal{K}} (u|_\Sigma - u_D) + \widetilde{\mathcal{V}} (\ssigma\cdot\n_\x + \phi_N) \quad \text{on } Q^{\rm ext}
\end{align}
are weak solutions of the transmission problem~\eqref{eq:transmission}.
From this and Lemma~\ref{lem:coercivity on kernel 4} in combination with Theorem~\ref{thm:coercivity}, we immediately deduce unique solvability of~\eqref{eq:transmission}.

\begin{theorem}\label{thm:equivalence}
If $u$ and $u^{\rm ext}$ are weak solutions of \eqref{eq:transmission}, then $(u,\ssigma) = (u,-\A\nabla_\x u) \in U$ satisfies \eqref{eq:transmission_interior} and \eqref{eq:transmission_exterior}, and $u^{\rm ext}$ is given by~\eqref{eq:2:interior2exterior}. 
Conversely, if $\u = (u,\ssigma) \in U$ solves \eqref{eq:transmission_interior} and \eqref{eq:transmission_exterior}, then $u$ and $u^{\rm ext}$ given by~\eqref{eq:2:interior2exterior} are weak solutions of \eqref{eq:transmission}.
Since \eqref{eq:transmission_interior} and \eqref{eq:transmission_exterior} are equivalent to the
(for sufficiently large $\alpha>0$) uniquely solvable variational formulation~\eqref{eq:transmission_weak4}, this shows that there is indeed a unique pair $u$ and $u^{\rm ext}$ of weak solutions of~\eqref{eq:transmission}. 
\end{theorem}

\begin{proof}
We have already seen that for any weak solution $u$ of \eqref{eq:transmission}, the 
corresponding $(u,-\A\nabla_\x u)\in U$ satisfies \eqref{eq:transmission_interior} and \eqref{eq:transmission_exterior}.
Moreover, \eqref{eq:2:interior2exterior} is just~\eqref{eq:interior2exterior} with $-\A\nabla_\x u$ replaced by the second variable $\ssigma$. 

Conversely, any $(u,\ssigma)\in U$ satisfies $u\in L^2(I;H^1(\Omega)) \cap H^1(I; H^{-1}(\Omega))$ according to \eqref{eq:lemma_gs21}, 
and the first-order system~\eqref{eq:transmission_interior} implies the heat equation in the interior~\eqref{eq:heat_interior} with initial condition~\eqref{eq:initial_interior}. 
Moreover, $u^{\rm ext}$ from~\eqref{eq:2:interior2exterior} satisfies the homogeneous heat equation in the exterior~\eqref{eq:heat_exterior} with initial condition~\eqref{eq:initial_exterior}. 
By restricting $u^{\rm ext}$ to the lateral boundary $\Sigma$, we see that 
\begin{align}\label{eq:aux:equivalence}
	u^{\rm ext}|_\Sigma = ({\mathcal{K}}+1/2) (u|_\Sigma - u_D) + {\mathcal{V}} (\ssigma\cdot\n_\x + \phi_N),
\end{align}
and \eqref{eq:transmission_exterior} shows the Dirichlet jump condition~\eqref{eq:jump_dirichlet}, i.e., $u|_\Sigma - u^{\rm ext}|_\Sigma = u_D$. 
According to \eqref{eq:representation_radiation}, $u^{\rm ext}$ satisfies the radiation condition~\eqref{eq:radiation}, from which we derive with the representation formula on $Q^{\rm ext}_R$ that $u^{\rm ext} = \widetilde{\mathcal{K}} (u^{\rm ext}|_\Sigma) - \widetilde{\mathcal{V}} (\nabla_\x u^{\rm ext}\cdot\n_\x)$; see Section~\ref{sec:weak_formulation}. 
By restricting this representation to $\Sigma$ and exploiting $u|_\Sigma - u^{\rm ext}|_\Sigma = u_D$, we see that 
\begin{align*}
	u^{\rm ext}|_\Sigma = (\mathcal{K}+1/2) (u|_\Sigma - u_D) - \mathcal{V} (\nabla_\x u^{\rm ext}\cdot\n_\x).
\end{align*}
By subtracting the latter equation from \eqref{eq:aux:equivalence}, we see that $0= \mathcal{V} ((\ssigma + \nabla_\x u^{\rm ext})\cdot \n_\x + \phi_N)$. 
Since $\mathcal{V}$ is an isomorphism and $\ssigma = -\A\nabla_\x u$, we conclude the Neumann jump condition~\eqref{eq:jump_neumann}, i.e., $(\A\nabla_\x u - \nabla_\x u^{\rm ext}) \cdot \n_\x =  \phi_N$, and thus the proof. 
\end{proof}

\begin{remark}\label{rem:equivalence}
The equivalence of Theorem~\ref{thm:equivalence} follows analogously if the boundary integral equation~\eqref{eq:transmission_exterior} is replaced by the boundary integral equation~\eqref{eq:transmission_exterior2}, which results from the application of the exterior normal derivative to the representation~\eqref{eq:interior2exterior} of the exterior solution $u^{\rm ext}$. 
\end{remark}

\section{Numerical experiments}\label{sec:numerics}
We present numerical experiments on the two-dimensional spatial domain $\Omega=(0,1)^2$ with final time $T=1$. 
In all experiments, $\A$ is the identity matrix. We use
the coupling approach involving the Dirichlet trace of the representation formula from
Section~\ref{sec:dirichlet_coupling}. To discretize the trial space
\begin{align*}
  U=\boldsymbol{H}(\div;Q) \cap (L^2(I;H^1(\Omega))\times L^2(Q)^d),
\end{align*}
we consider a partition $I_{h_t}$ of $I$ into intervals of size $h_t$ and a conforming partition $\Omega_{h_{\x}}$ of $\Omega$ into shape-regular triangles of size $h_\x$.
In particular, this gives rise to a partition
\begin{align*}
  Q_{h} := \left\{ J\times K\mid J\in I_{h_t}, K\in \Omega_{h_{\x}} \right\}
\end{align*}
of $Q=I\times \Omega$ into prisms. 
In time, we use the space $S^1(I_{h_t})$ of globally continuous $I_{h_t}$-piecewise affine functions for
the first variable $u$ and the space $P^0(I_{h_t})$ of $I_{h_t}$-piecewise constant functions
for the second variable $\ssigma$.
In the spatial direction, we use the space $S^1(\Omega_{h_{\x}})$ of globally continuous $\Omega_{h_{\x}}$-piecewise
affine functions for the first variable $u$ and the space $RT^0(\Omega_{h_{\x}})$ of lowest-order Raviart--Thomas functions
for the second variable $\ssigma$.
As in~\cite{gs24}, we use the resulting tensor-product space
\begin{align*}
  U_{h} := [S^1(I_{h_t})\otimes S^1(\Omega_{h_{\x}})] \times [P^0(I_{h_t})\otimes RT^0(\Omega_{h_{\x}})] \subset U
\end{align*}
as conforming approximation space of $U$.
For given data $(f,u_0,u_D,\phi_N)$, the discrete solution $\u_h = (u_h,\ssigma_h)\in U_h$ is then given by
\begin{align*}
\begin{split}
	b_\alpha(\u_h,\v_h) + c_1(\u_h,\v_h) 
	= \alpha \big[\dual{f}{\div \v}_{L^2(Q)} + \dual{u_0}{v(0,\cdot)}_{L^2(\Omega)}\big] 
	\\ 
		 + \dual{-\mathcal{V} \phi_N + (\mathcal{K}-\tfrac12) u_D}{\ttau_h\cdot\n_\x}_\Sigma
\end{split}
\end{align*}
for all $\v_h = (v_h,\ttau_h)\in U_h$.
Lemma~\ref{lem:coercivity on kernel} and Theorem~\ref{thm:coercivity} yield coercivity of the bilinear form on the left-hand side provided that $\alpha$
is chosen sufficiently large, which guarantees unique solvability of this discrete variational formulation.
In our experiments, we choose $\alpha=1$. The numerical eigenvalue study provided in Section~\ref{sec:dependence_alpha}
indeed suggests that this induces a coercive bilinear form. 

More precisely, we consider the following concrete sequence of uniform meshes on $Q=(0,1)^3$: For given $\ell\in\mathbb{N}$,
we first split the spatial domain $\Omega=(0,1)^2$ in each direction into $2^\ell-1$ uniform intervals.
The resulting squares are then bisected in northwest-southeast direction
to obtain a conforming shape-regular triangulation $\Omega_{h_{\x}}$. Finally, the time interval $I=(0,1)$ is split into $2^\ell-1$ uniform intervals, forming $I_{h_t}$.
With the number $N_h=\dim (S^1(I_k)\otimes S^1(\Omega_h))$ of nodes in $Q_{h}$ and the number of overall degrees of freedom $M_h:=\dim U_h$, this yields that
\begin{align*}
  h:=h_t\eqsim h_{\x} \eqsim N_h^{-1/3} \eqsim M_h^{-1/3}.
\end{align*}

For $\alpha$ sufficiently large such that $b_\alpha+c_1$ is coercive, the C\'ea lemma  implies quasi-optimality of the method.
That is, for the exact solution $(u,\ssigma)\in U$ and numerical approximation $(u_h,\ssigma_h)\in U_h$,
it holds that
\begin{align*}
  \| (u,\ssigma)-(u_h,\ssigma_h) \|_U \eqsim \min_{(v_h,\ttau_h) \in U_h} \| (u,\ssigma)-(v_h,\ttau_h) \|_U.
\end{align*}
For smooth exact solutions,~\cite{gs24} shows that
\begin{align*}
  \min_{(v_h,\ttau_h)\in U_h} \| (u,\ssigma)-(v_h,\ttau_h) \|_U =\mathcal{O}(h).
\end{align*}

The experiments are implemented in \texttt{MATLAB}. For the calculation of the discrete boundary
integral operators, we have analytically integrated integrals in time, cf.~\cite{costabel90,reinarz15,gv22},
which can be done up to evaluations of exponential integrals. The latter can be evaluated directly in \texttt{MATLAB}.
The remaining integrals in space are calculated either by tensor Gauss--Legendre rules or adapted quadrature
rules from~\cite{smith00} in case of (logarithmic) singularities.
All systems are solved by the \texttt{MATLAB} backslash operator without preconditioning.
\subsection{Smooth solutions vanishing in the exterior domain}\label{sec:exterior0}
We prescribe an exact interior and exterior solution as
\begin{align*}
  u(t,\x) &= G(t,\x-\x^{\rm ext}),\\
  u^{\rm ext}(t,\x) &= 0,
\end{align*}
with $\x^{\rm ext} = (-0.5,-0.5)\in\Omega^{\rm ext}$, and compute the data thereof. In particular, $f=0$, $u_0=0$, and $u_D$ and $\phi_N$ are smooth.
On a uniform sequence of meshes with $N_h$ nodes, we expect
\begin{align*}
  \| (u,\ssigma)-(u_h,\ssigma_h) \|_U^2
  = \OO(h^2) = \OO(N_h^{-2/3}).
\end{align*}
In Figure~\ref{fig:ex1}, we plot the (squared) errors
\begin{align*}
  \| \div(\u-\u_h) \|_{L^2(Q)}^2, \qquad
  \| u - u_h \|_{L^2(Q)}^2,\qquad
  \| \nabla_\x (u - u_h) \|_{L^2(Q)}^2,\qquad
  \| \ssigma - \ssigma_h \|_{L^2(Q)}^2.
\end{align*}
It is reasonable to expect higher rates for the lower-order term
$\| u - u_h \|_{L^2(Q)}^2=\OO(h^4) = \OO(N_h^{-4/3})$
and optimal order $\OO(N_h^{-2/3})$ for the three remaining error terms.
The rates are confirmed in Figure~\ref{fig:ex1}.
\begin{figure}[h]
  \begin{center}
      \begin{tikzpicture}
\begin{loglogaxis}[
width=0.65\textwidth,
cycle list/Dark2-6,
cycle multiindex* list={
mark list*\nextlist
Dark2-6\nextlist},
every axis plot/.append style={ultra thick},
xlabel={number of nodes $N_h$},
grid=major,
legend entries={
  \small $\| \div(\u-\u_h) \|_{L^2(Q)}^2$,
  \small $\| u - u_h \|_{L^2(Q)}^2$,
  \small $\| \nabla_\x(u - u_h) \|_{L^2(Q)}^2$,
\small $\| \ssigma - \ssigma_h \|_{L^2(Q)}^2$},
legend pos=south west,
]
\addplot table [x=dofsh,y=e_usigmah_Hdiv] {numerics/ex1.dat};
\addplot table [x=dofsh,y=e_uh_L2L2] {numerics/ex1.dat};
\addplot table [x=dofsh,y=e_uh_L2H1] {numerics/ex1.dat};
\addplot table [x=dofsh,y=e_sigmah_L2] {numerics/ex1.dat};

\logLogSlopeTriangle{0.9}{0.6}{0.6}{2/3}{black}{{\small $2/3$}};
\logLogSlopeTriangle{0.9}{0.3}{0.15}{4/3}{black}{{\small $4/3$}};
\end{loglogaxis}
\end{tikzpicture}
  \end{center}
  \caption{Smooth solution vanishing in the exterior domain.}
  \label{fig:ex1}
\end{figure}
\subsection{Smooth solutions vanishing in the interior domain}
We prescribe an exact interior and exterior solution as
\begin{align*}
  u(t,\x) &= 0,\\
  u^{\rm ext}(t,\x) &= G(t,\x-\x^{\rm int}),
\end{align*}
with $\x^{\rm int} = (0.5,0.5)\in\Omega$, and compute the data thereof.
In particular, $f=0$, $u_0=0$, and $u_D$ and $\phi_N$ are smooth.
In Figure~\ref{fig:ex2}, we plot the same errors as in Section~\ref{sec:exterior0}, expecting similar rates.
\begin{figure}[h]
  \begin{center}
      \begin{tikzpicture}
\begin{loglogaxis}[
width=0.65\textwidth,
cycle list/Dark2-6,
cycle multiindex* list={
mark list*\nextlist
Dark2-6\nextlist},
every axis plot/.append style={ultra thick},
xlabel={number of nodes $N_h$},
grid=major,
legend entries={
  \small $\| \div(\u-\u_h) \|_{L^2(Q)}^2$,
  \small $\| u - u_h \|_{L^2(Q)}^2$,
  \small $\| \nabla_\x(u - u_h) \|_{L^2(Q)}^2$,
  \small $\| \ssigma - \ssigma_h \|_{L^2(Q)}^2$},
legend pos=south west,
]
\addplot table [x=dofsh,y=e_usigmah_Hdiv] {numerics/ex2.dat};
\addplot table [x=dofsh,y=e_uh_L2L2] {numerics/ex2.dat};
\addplot table [x=dofsh,y=e_uh_L2H1] {numerics/ex2.dat};
\addplot table [x=dofsh,y=e_sigmah_L2] {numerics/ex2.dat};

\logLogSlopeTriangle{0.9}{0.6}{0.68}{2/3}{black}{{\small $2/3$}};
\logLogSlopeTriangleBelow{0.6}{0.3}{0.4}{4/3}{black}{{\small $4/3$}};
\end{loglogaxis}
\end{tikzpicture}
  \end{center}
  \caption{Smooth solution vanishing in the interior domain.}
  \label{fig:ex2}
\end{figure}
\subsection{Smooth solutions}
We prescribe an exact interior and exterior solution as
\begin{align*}
  u(t,\x) &= x_1^2 + t,\\
  u^{\rm ext}(t,\x) &= G(t,\x-\x^{\rm int}),
\end{align*}
with $\x^{\rm int} = (0.5,0.5)\in\Omega$, and compute the data thereof. In particular,
$f=-1$, $u_0(\x)=x_1^2$, and $u_D$ and $\phi_N$ are smooth.
In Figure~\ref{fig:ex3}, we plot the same errors as in {Section~\ref{sec:exterior0}, expecting similar rates.
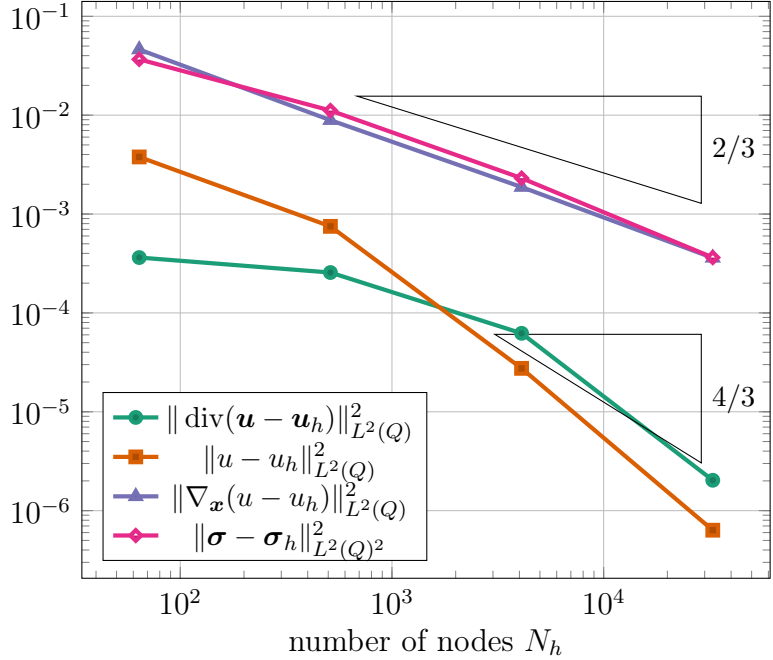
\begin{figure}[h]
  \begin{center}
      \begin{tikzpicture}
\begin{loglogaxis}[
width=0.65\textwidth,
cycle list/Dark2-6,
cycle multiindex* list={
mark list*\nextlist
Dark2-6\nextlist},
every axis plot/.append style={ultra thick},
xlabel={number of nodes $N_h$},
grid=major,
legend entries={
  \small $\| \div(\u-\u_h) \|_{L^2(Q)}^2$,
  \small $\| u - u_h \|_{L^2(Q)}^2$,
  \small $\| \nabla_\x(u - u_h) \|_{L^2(Q)}^2$,
  \small $\| \ssigma - \ssigma_h \|_{L^2(Q)^2}^2$},
legend pos=south west,
]
\addplot table [x=dofsh,y=e_usigmah_Hdiv] {numerics/ex3.dat};
\addplot table [x=dofsh,y=e_uh_L2L2] {numerics/ex3.dat};
\addplot table [x=dofsh,y=e_uh_L2H1] {numerics/ex3.dat};
\addplot table [x=dofsh,y=e_sigmah_L2] {numerics/ex3.dat};

\logLogSlopeTriangle{0.9}{0.5}{0.65}{2/3}{black}{{\small $2/3$}};
\logLogSlopeTriangle{0.9}{0.3}{0.2}{4/3}{black}{{\small $4/3$}};
\end{loglogaxis}
\end{tikzpicture}
  \end{center}
  \caption{Smooth solution.}
  \label{fig:ex3}
\end{figure}

\subsection{Nonsmooth solutions}
We use the data $f,u_D,\phi_N=0$ and $u_0=1$.
The pair of exact solutions $(u,u^{\rm ext})$ is unknown. Since the initial conditions do not match with
the Dirichlet transmission conditions, we expect singularities in the solutions leading
to sub-optimal convergence rates. In Figure~\ref{fig:ex4}, we plot
\begin{align*}
	\| \mathcal{G}(\u-\u_h) \|_{L}^2 &= \| (f,0,u_0) - \mathcal{G}\u_h \|_L^2 
	\\
  	&= \norm{f - \div \u_h}{L^2(Q)}^2 + \norm{\A\nabla_\x u_h + \ssigma_h}{L^2(Q)}^2 + \norm{u_0 - u_h(0,\cdot)}{L^2(\Omega)}^2,
\end{align*}
which constitutes a computable lower bound of the (squared) error $\norm{\u-\u_h}{U}^2$.
We observe that this error measure converges approximately with rate
$N_h^{-1/5}$.
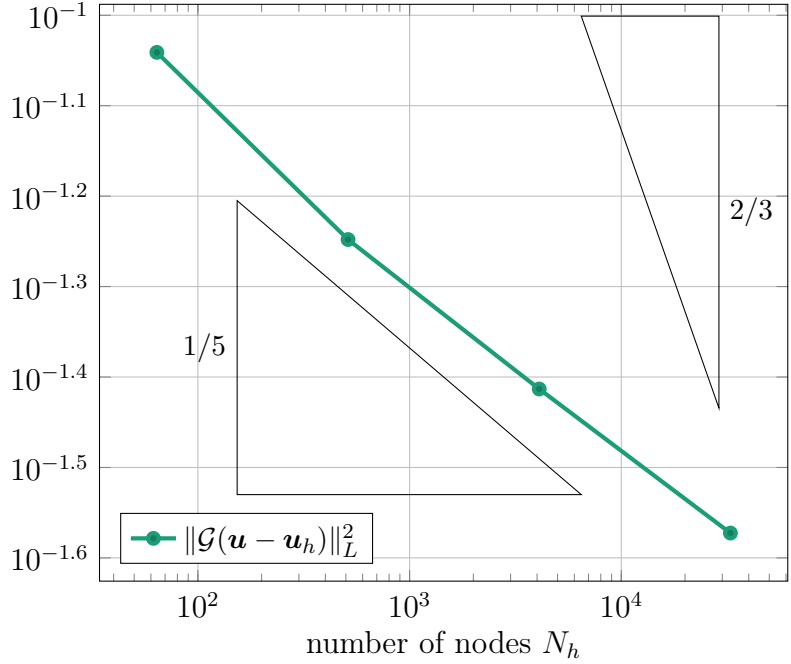
\begin{figure}[h]
  \begin{center}
      \begin{tikzpicture}
\begin{loglogaxis}[
width=0.65\textwidth,
cycle list/Dark2-6,
cycle multiindex* list={
mark list*\nextlist
Dark2-6\nextlist},
every axis plot/.append style={ultra thick},
xlabel={number of nodes $N_h$},
grid=major,
legend entries={
  \small $\| \mathcal{G}(\u-\u_h) \|_{L}^2$},
legend pos=south west,
]
\addplot table [x=dofsh,y=e_eq] {numerics/ex4.dat};

\logLogSlopeTriangle{0.9}{0.2}{0.3}{2/3}{black}{{\small $2/3$}};
\logLogSlopeTriangleBelow{0.7}{0.5}{0.15}{2/10}{black}{{\small $1/5$}};
\end{loglogaxis}
\end{tikzpicture}
  \end{center}
  \caption{Nonsmooth solution.}
  \label{fig:ex4}
\end{figure}

\subsection{Dependence of coercivity on $\boldsymbol{\alpha}$} \label{sec:dependence_alpha}
In this section, we investigate the influence of $\alpha$ on the coercivity of the bilinear form $b_\alpha+c_1$.
If $\alpha$ is sufficiently large, Lemma~\ref{lem:coercivity on kernel} and Theorem~\ref{thm:coercivity} guarantee the existence of a constant $\const{coer}(\alpha)>0$ such that
{\begin{align}\label{eq:coercivity1}
	b_\alpha(\v,\v) + c_1(\v,\v) \ge \const{coer}(\alpha) \norm{\v}{U}^2 \quad \text{for all }\v\in U.
\end{align}
For the discrete subspace $U_h$ of dimension $M_h$, 
let $\textbf{A}_{\alpha,h}\in \R^{M_h\times M_h}$ be the system matrix of the bilinear form
$b_\alpha+c_1$, and $\textbf{U}_h\in \R^{M_h\times M_h}$ be the (symmetric and positive definite) system matrix of the inner product on $U$ with respect to some fixed basis of $U_h$.
With the symmetric part $\textbf{A}_{\alpha,h}^{\rm s} := (\textbf{A}_{\alpha,h} + \textbf{A}_{\alpha,h}^\top)/2$  of $\textbf{A}_{\alpha,h}$, it holds that
\begin{align*}
\inf_{\v_h \in U_h} \frac{b_\alpha(\v_h,\v_h)+c_1(\v_h,\v_h)}{\norm{\v_h}{U}^2}
  =\inf_{\textbf{y}\in\R^{M_h}}\frac{\textbf{y}^\top \textbf{A}_{\alpha,h}\textbf{y}}{\textbf{y}^\top \textbf{U}_h \textbf{y}}
  = \inf_{\textbf{y}\in\R^{M_h}}\frac{\textbf{y}^\top \textbf{A}_{\alpha,h}^{\rm s} \textbf{y}}{\textbf{y}^\top \textbf{U}_h \textbf{y}}
  = \inf_{\textbf{y}\in\R^{M_h}} \frac{\textbf{y}^\top \textbf{U}_h^{-1/2} \textbf{A}_{\alpha,h}^{\rm s} \textbf{U}_{h}^{-1/2}\textbf{y}}{\textbf{y}^\top \textbf{y}}.
\end{align*}
As $\textbf{U}_h^{-1/2} \textbf{A}_{\alpha,h}^{\rm s} \textbf{U}_{h}^{-1/2}$ is symmetric, the infimum on the
right-hand side is equal to the smallest eigenvalue of this matrix. Hence, we consider the equivalent generalized eigenvalue problem
\begin{align*}
  \textbf{A}_{\alpha,h}^{\rm s} \textbf{y} = \lambda \textbf{U}_{h} \textbf{y}.
\end{align*}
In Table~\ref{tab:coercivity}, we provide the smallest eigenvalue 
in function
of $\alpha$ (columns) and meshes of different size (rows).
\begin{table}[h]
  \begin{center}
  \begin{tabular}{r | r | r | r | r | r | r | r | r| r | r |}
    $N_h\textbackslash \alpha$ & 10000 & 1000 & 100 & 10 & 1 & 0.1 & 0.01 & 0.001 & 0.0001\\
    \hline\hline
    64 & 89.1751 & 10.0269 &   1.6216 &   0.6109 &   0.0593 &  -0.7387 &  -1.1230 &  -1.1641 &  -1.1682\\
    512 & 22.3626 &  3.0514 &   0.7684 &   0.3073 &   0.0223 &  -0.7546 &  -1.1696 &  -1.2150 &  -1.2195\\
    4096 & 6.4107 & 1.1080 &   0.5000 &   0.2201 &   0.0124 &  -0.7617 &  -1.1906 &  -1.2378 &  -1.2426\\
    32768 &   NaN & 0.6609 &   0.2917 &   0.1907 &   0.0091 &  -0.7664 &  -1.2072 &  -1.2561 &  -1.2611
  \end{tabular}
  \end{center}
  \caption{Minimum of $\frac{b_\alpha(\v_h,\v_h)+c_1(\v_h,\v_h)}{\norm{\v_h}{U}^2}$ over all $\v_h\in U_h$.}
  \label{tab:coercivity}
\end{table}
We see that for $\alpha\in\left\{10000, 1000, 100,10,1 \right\}$, all eigenvalues are positive and appear to be bounded from below by
some positive constant.
For $\alpha\in \left\{ 0.1, 0.01, 0.001, 0.0001 \right\}$, the smallest eigenvalue is already negative.
\section*{Acknowledgement}
TF acknowledges funding by ANID Chile through FONDECYT project 1210391 \& project 1250070.
GG acknowledges funding by the Deutsche Forschungsgemeinschaft (DFG, German Research Foundation)
under Germany's Excellence Strategy – EXC-2047/1 – 390685813.
MK acknowledges funding by ANID Chile through FONDECYT project 1210579 \& project 1250604.


\bibliographystyle{siamplain}
\bibliography{literature}

\begin{thebibliography}{10}

\bibitem{an87}
{\sc D.~N. Arnold and P.~J. Noon}, {\em Boundary integral equations of the
  first kind for the heat equation}, in Boundary elements IX, vol.~3, Springer,
  1987, pp.~213--229.

\bibitem{affkmp13}
{\sc M.~Aurada, M.~Feischl, T.~F{\"u}hrer, M.~Karkulik, J.~M. Melenk, and
  D.~Praetorius}, {\em Classical {FEM-BEM} coupling methods: nonlinearities,
  well-posedness, and adaptivity}, Comput. Mech., 51 (2013), pp.~399--419.

\bibitem{bf12}
{\sc F.~Boyer and P.~Fabrie}, {\em Mathematical tools for the study of the
  incompressible {Navier-Stokes} equations and related models}, Springer
  Science \& Business Media, New York, 2012.

\bibitem{bdfgl24}
{\sc E.~Burman, R.~Durst, M.~A. Fern{\'a}ndez, J.~Guzm{\'a}n, and S.~Liu}, {\em
  A second-order correction method for loosely coupled discretizations applied
  to parabolic--parabolic interface problems}, IMA J. Numer. Anal.,  (2024),
  p.~drae075.

\bibitem{ck97}
{\sc R.~Chapko and R.~Kress}, {\em Rothe's method for the heat equation and
  boundary integral equations}, J. Integral Equations Appl., 9 (1997),
  pp.~47--69.

\bibitem{cz98}
{\sc Z.~Chen and J.~Zou}, {\em Finite element methods and their convergence for
  elliptic and parabolic interface problems}, Numer. Math., 79 (1998),
  pp.~175--202.

\bibitem{costabel90}
{\sc M.~Costabel}, {\em Boundary integral operators for the heat equation},
  Integral Equations Operator Theory, 13 (1990), pp.~498--552.

\bibitem{ces90}
{\sc M.~Costabel, V.~J. Ervin, and E.~P. Stephan}, {\em Symmetric coupling of
  finite elements and boundary elements for a parabolic-elliptic interface
  problem}, Quart. Appl. Math., 48 (1990), pp.~265--279.

\bibitem{dl92}
{\sc R.~Dautray and J.-L. Lions}, {\em {Mathematical analysis and numerical
  methods for science and technology: Volume 5 Evolution problems I}},
  Springer, Berlin, 1992.

\bibitem{dohr19}
{\sc S.~Dohr}, {\em Distributed and preconditioned space--time boundary element
  methods for the heat equation}, PhD thesis, Graz University of Technology,
  2019.

\bibitem{dns19}
{\sc S.~Dohr, K.~Niino, and O.~Steinbach}, {\em Space-time boundary element
  methods for the heat equation}, in Space-Time Methods: Applications to
  Partial Differential Equations, De Gruyter, 2019, pp.~1--60.

\bibitem{dzomk19}
{\sc S.~Dohr, J.~Zapletal, G.~Of, M.~Merta, and M.~Kravcenko}, {\em A parallel
  space--time boundary element method for the heat equation}, Comput. Math.
  Appl., 78 (2019), pp.~2852--2866.

\bibitem{dubach96}
{\sc E.~Dubach}, {\em Artificial boundary conditions for diffusion equations:
  numerical study}, J. Comput. Appl. Math., 70 (1996), pp.~127--144.

\bibitem{ees18}
{\sc H.~Egger, C.~Erath, and R.~Schorr}, {\em On the nonsymmetric coupling
  method for parabolic-elliptic interface problems}, SIAM J. Numer. Anal., 56
  (2018), pp.~3510--3533.

\bibitem{es20}
{\sc C.~Erath and R.~Schorr}, {\em Stable non-symmetric coupling of the finite
  volume method and the boundary element method for convection-dominated
  parabolic-elliptic interface problems}, Comput. Methods Appl. Math., 20
  (2020), pp.~251--272.

\bibitem{eg21a}
{\sc A.~Ern and J.-L. Guermond}, {\em {Finite elements I: Approximation and
  interpolation}}, vol.~72, Springer Nature, Cham, 2021.

\bibitem{evans10}
{\sc L.~C. Evans}, {\em Partial differential equations}, vol.~19 of Graduate
  Studies in Mathematics, American Mathematical Society, Providence,
  second~ed., 2010.

\bibitem{ffps17}
{\sc M.~Feischl, T.~F{\"u}hrer, D.~Praetorius, and E.~P. Stephan}, {\em Optimal
  preconditioning for the symmetric and nonsymmetric coupling of adaptive
  finite elements and boundary elements}, Numer. Methods Partial Differential
  Equations, 33 (2017), pp.~603--632.

\bibitem{folland95}
{\sc G.~B. Folland}, {\em Introduction to partial differential equations},
  Princeton University Press, Princeton, NJ, second~ed., 1995.

\bibitem{fr17}
{\sc S.~Frei and T.~Richter}, {\em A second order time-stepping scheme for
  parabolic interface problems with moving interfaces}, ESAIM Math. Model.
  Numer. Anal., 51 (2017), pp.~1539--1560.

\bibitem{fgk23}
{\sc T.~F{\"u}hrer, R.~Gonz{\'a}lez, and M.~Karkulik}, {\em Well-posedness of
  first-order acoustic wave equations and space-time finite element
  approximation}, Preprint, arXiv:2311.10536 (2023).

\bibitem{fh17}
{\sc T.~F{\"u}hrer and N.~Heuer}, {\em Robust coupling of {DPG} and {BEM} for a
  singularly perturbed transmission problem}, Comput. Math. Appl., 74 (2017),
  pp.~1940--1954.

\bibitem{fhk17}
{\sc T.~F{\"u}hrer, N.~Heuer, and M.~Karkulik}, {\em {On the coupling of DPG
  and BEM}}, Math. Comp., 86 (2017), pp.~2261--2284.

\bibitem{fk21}
{\sc T.~F\"uhrer and M.~Karkulik}, {\em Space--time least-squares finite
  elements for parabolic equations}, Comput. Math. Appl., 92 (2021),
  pp.~27--36.

\bibitem{fk23}
{\sc T.~F{\"u}hrer and M.~Karkulik}, {\em Least-squares finite elements for
  distributed optimal control problems}, Numer. Math., 154 (2023),
  pp.~409--442.

\bibitem{fk24}
{\sc T.~F{\"u}hrer and M.~Karkulik}, {\em Space-time least-squares finite
  element methods for parabolic distributed optimal control problems}, Comput.
  Methods Appl. Math., 24 (2024), pp.~673--691.

\bibitem{gn16}
{\sc M.~J. Gander and M.~Neum\"uller}, {\em Analysis of a new space-time
  parallel multigrid algorithm for parabolic problems}, SIAM J. Sci. Comput.,
  38 (2016), pp.~A2173--A2208.

\bibitem{gs21}
{\sc G.~Gantner and R.~Stevenson}, {\em Further results on a space-time {FOSLS}
  formulation of parabolic {PDEs}}, ESAIM Math. Model. Numer. Anal., 55 (2021),
  pp.~283--299.

\bibitem{gs22}
{\sc G.~Gantner and R.~Stevenson}, {\em A well-posed first order system least
  squares formulation of the instationary {S}tokes equations}, SIAM J. Numer.
  Anal., 60 (2022), pp.~1607--1629.

\bibitem{gs23}
{\sc G.~Gantner and R.~Stevenson}, {\em {Applications of a space-time FOSLS
  formulation for parabolic PDEs}}, IMA J. Numer. Anal., 44 (2023), pp.~58--82.

\bibitem{gs24}
{\sc G.~Gantner and R.~Stevenson}, {\em {Improved rates for a space--time FOSLS
  of parabolic PDEs}}, Numer. Math., 156 (2024), pp.~133--157.

\bibitem{gv22}
{\sc G.~Gantner and R.~van Veneti{\"e}}, {\em Adaptive space-time {BEM} for the
  heat equation}, Comput. Math. Appl., 107 (2022), pp.~117--131.

\bibitem{guo21}
{\sc R.~Guo}, {\em Solving parabolic moving interface problems with dynamical
  immersed spaces on unfitted meshes: fully discrete analysis}, SIAM J. Numer.
  Anal., 59 (2021), pp.~797--828.

\bibitem{hh02}
{\sc H.~Han and Z.~Huang}, {\em A class of artificial boundary conditions for
  heat equation in unbounded domains}, Comput. Math. Appl., 43 (2002),
  pp.~889--900.

\bibitem{ht18}
{\sc H.~Harbrecht and J.~Tausch}, {\em A fast sparse grid based space--time
  boundary element method for the nonstationary heat equation}, Numer. Math.,
  140 (2018), pp.~1--26.

\bibitem{hllz13}
{\sc X.~He, T.~Lin, Y.~Lin, and X.~Zhang}, {\em Immersed finite element methods
  for parabolic equations with moving interface}, Numer. Methods Partial
  Differential Equations, 29 (2013), pp.~619--646.

\bibitem{lg07}
{\sc J.-R. Li and L.~Greengard}, {\em On the numerical solution of the heat
  equation. {I}. {F}ast solvers in free space}, J. Comput. Phys., 226 (2007),
  pp.~1891--1901.

\bibitem{lm72b}
{\sc J.-L. Lions and E.~Magenes}, {\em Non-homogeneous boundary value problems
  and applications. {V}olume {II}}, Springer-Verlag, New York-Heidelberg, 1972.

\bibitem{lubich88}
{\sc C.~Lubich}, {\em Convolution quadrature and discretized operational
  calculus. {I}}, Numer. Math., 52 (1988), pp.~129--145.

\bibitem{lo93}
{\sc C.~Lubich and A.~Ostermann}, {\em Runge-{K}utta methods for parabolic
  equations and convolution quadrature}, Math. Comp., 60 (1993), pp.~105--131.

\bibitem{ls92}
{\sc C.~Lubich and R.~Schneider}, {\em Time discretization of parabolic
  boundary integral equations}, Numer. Math., 63 (1992), pp.~455--481.

\bibitem{ms87}
{\sc R.~C. MacCamy and M.~Suri}, {\em A time-dependent interface problem for
  two-dimensional eddy currents}, Quart. Appl. Math., 44 (1987), pp.~675--690.

\bibitem{mst14}
{\sc M.~Messner, M.~Schanz, and J.~Tausch}, {\em A fast {G}alerkin method for
  parabolic space--time boundary integral equations}, J. Comput. Phys., 258
  (2014), pp.~15--30.

\bibitem{mst15}
{\sc M.~Messner, M.~Schanz, and J.~Tausch}, {\em An efficient {G}alerkin
  boundary element method for the transient heat equation}, SIAM J. Sci.
  Comput., 37 (2015), pp.~A1554--A1576.

\bibitem{ms98}
{\sc P.~Mund and E.~P. Stephan}, {\em {The preconditioned GMRES method for
  systems of coupled FEM-BEM equations}}, Adv. Comput. Math., 9 (1998),
  pp.~131--144.

\bibitem{ns19}
{\sc M.~Neum{\"u}ller and I.~Smears}, {\em Time-parallel iterative solvers for
  parabolic evolution equations}, SIAM J. Sci. Comput., 41 (2019),
  pp.~C28--C51.

\bibitem{noon88}
{\sc P.~J. Noon}, {\em The single layer heat potential and Galerkin boundary
  element methods for the heat equation}, PhD thesis, University of Maryland,
  College Park, 1988.

\bibitem{os13}
{\sc G.~Of and O.~Steinbach}, {\em Is the one-equation coupling of finite and
  boundary element methods always stable?}, ZAMM Z. Angew. Math. Mech., 93
  (2013), pp.~476--484.

\bibitem{qrsz19}
{\sc T.~Qiu, A.~Rieder, F.-J. Sayas, and S.~Zhang}, {\em Time-domain boundary
  integral equation modeling of heat transmission problems}, Numer. Math., 143
  (2019), pp.~223--259.

\bibitem{rs07}
{\sc M.-L. Rap{\'u}n and F.-J. Sayas}, {\em Boundary element simulation of
  thermal waves}, Arch. Comput. Methods Eng., 14 (2007), pp.~3--46.

\bibitem{reinarz15}
{\sc A.~Reinarz}, {\em Sparse space-time boundary element methods for the heat
  equation}, PhD thesis, University of Reading, 2015.

\bibitem{rsm20}
{\sc A.~Rieder, F.-J. Sayas, and J.~M. Melenk}, {\em Runge--{K}utta
  approximation for {$C_0$}-semigroups in the graph norm with applications to
  time domain boundary integral equations}, Partial Differ. Equ. Appl., 1
  (2020), p.~49.

\bibitem{sayas09}
{\sc F.-J. Sayas}, {\em {The validity of Johnson--N{\'e}d{\'e}lec's BEM--FEM
  coupling on polygonal interfaces}}, SIAM J. Numer. Anal., 47 (2009),
  pp.~3451--3463.

\bibitem{ss09}
{\sc C.~Schwab and R.~Stevenson}, {\em Space-time adaptive wavelet methods for
  parabolic evolution problems}, Math. Comp., 78 (2009), pp.~1293--1318.

\bibitem{smith00}
{\sc R.~Smith}, {\em {Direct Gauss quadrature formulae for logarithmic
  singularities on isoparametric elements}}, {Eng. Anal. Bound. Elem.}, 24
  (2000), pp.~161--167.

\bibitem{steinbach11}
{\sc O.~Steinbach}, {\em A note on the stable one-equation coupling of finite
  and boundary elements}, SIAM J. Numer. Anal., 49 (2011), pp.~1521--1531.

\bibitem{steinbach15}
{\sc O.~Steinbach}, {\em Space-time finite element methods for parabolic
  problems}, Comput. Methods Appl. Math., 15 (2015), pp.~551--566.

\bibitem{sz20}
{\sc O.~Steinbach and M.~Zank}, {\em Coercive space-time finite element methods
  for initial boundary value problems}, Electron. Trans. Numer. Anal., 52
  (2020), pp.~154--194.

\bibitem{sd11}
{\sc A.~Y. Suhov and A.~Ditkowski}, {\em Artificial boundary conditions for the
  simulation of the heat equation in an infinite domain}, SIAM J. Sci. Comput.,
  33 (2011), pp.~1765--1784.

\bibitem{vw20}
{\sc R.~van Veneti{\"e} and J.~Westerdiep}, {\em A parallel algorithm for
  solving linear parabolic evolution equations}, in Workshops on
  Parallel-in-Time Integration, Springer, 2020, pp.~33--50.

\bibitem{wmoz22}
{\sc R.~Watschinger, M.~Merta, G.~Of, and J.~Zapletal}, {\em A parallel fast
  multipole method for a space-time boundary element method for the heat
  equation}, SIAM J. Sci. Comput., 44 (2022), pp.~C320--C345.

\bibitem{zwom21}
{\sc J.~Zapletal, R.~Watschinger, G.~Of, and M.~Merta}, {\em Semi-analytic
  integration for a parallel space-time boundary element method modelling the
  heat equation}, Computers \& Mathematics with Applications, 103 (2021),
  pp.~156--170.

\bibitem{zx23}
{\sc C.~Zheng and J.~Xie}, {\em Fast artificial boundary method for the heat
  equation on unbounded domains with strip tails}, J. Comput. Appl. Math., 425
  (2023), p.~115032.

\end{thebibliography}

\end{document}